\documentclass[12pt]{amsart}
\usepackage{amssymb,latexsym}
\usepackage[mathscr]{eucal}
\usepackage{graphics}

\theoremstyle{plain}
\newtheorem{theorem}{Theorem}[section]
\newtheorem{corollary}[theorem]{Corollary}
\newtheorem{lemma}[theorem]{Lemma}
\newtheorem{proposition}[theorem]{Proposition}
\newtheorem{remark}[theorem]{Remark}
\newtheorem{definition}[theorem]{Definition}
\newtheorem{example}[theorem]{Example}

\begin{document}

\title{Traces on operator ideals
and arithmetic means}
\author{Victor Kaftal} 
\address{University of Cincinnati\\ 
         Department of Mathematics\\
         Cincinnati, OH, 45221-0025\\
         USA} 
\email{victor.kaftal@uc.edu}
\author[Gary Weiss]{Gary Weiss*}\thanks{*The second named author was partially supported by NSF Grants DMS 95-03062
     and DMS 97-06911 and both authors by grants from The Charles Phelps Taft Research Center.}
\email{gary.weiss@uc.edu}


\begin{abstract}
This article investigates the codimension of commutator spaces $[I, B(H)]$ of operator ideals on a separable Hilbert space,
i.e., ``How many traces can an ideal support?"
We conjecture that the codimension can be only zero, one, or infinity.
Using the arithmetic mean (am) operations on ideals introduced in \cite{DFWW}
and the analogous arithmetic mean operations at infinity (am-$\infty$) that we develop extensively in this article,
the conjecture is proven for all ideals not contained in the largest am-$\infty$ stable ideal and not containing the smallest am-stable ideal,
for all soft-edged ideals (i.e., $I=se(I) = IK(H)$) and all soft-complemented ideals \linebreak
(i.e., $I=scI= I/K(H)$), 
which include all classical operator ideals we considered.
The conjecture is proven also for ideals whose soft-interior, $seI$, or soft-cover, $scI$, are am-$\infty$ stable or am-stable and for other classes of ideals.
We show that an ideal of trace class operators supports a unique trace (up to scalar multiples) if and only if it is am-$\infty$ stable.
For a principal ideal, am-$\infty$ stability is what we call regularity at infinity of the sequence of s-numbers of the generator.
We prove for these sequences analogs of several of the characterizations of usual regularity.
In the process we apply trace extension methods to two problems on elementary operators studied by V. Shulman.
This article presents and extends several of the results announced in PNAS-US \cite{vKgW02} and then expanded and developed in a series of papers  \cite{vKgW07-soft}-\cite{vKgW07-Majorization}. 
\end{abstract}

\maketitle

\section{\leftline {\bf Introduction}}

The study of operator ideals - two-sided ideals of the algebra $B(H)$ of bounded linear operators on a complex separable infinite dimensional Hilbert space $H$ - started with J. Calkin \cite{jC41} in 1941. 
From early on (e.g., \cite{pH54}, \cite{BP65}, \cite{PT71} and \cite{jA77}), 
central in this area was the notion of commutator space and the related notion of trace. 
The commutator space (or commutator ideal) $[I,B(H)]$ of an ideal $I$ is the linear span of the commutators of operators in $I$ with operators in $B(H)$.
A trace on an operator ideal is a linear functional (not necessarily positive) that vanishes on its commutator space 
or, equivalently, that is unitarily invariant.  

The introduction of cyclic cohomology in the early 1980's by 
A. Connes and its linkage to algebraic K-theory by 
M. Wodzicki in the 1990's 
provided additional motivation for the complete determination of the structure of commutator spaces.
(Cf. \cite{aC82}, \cite{aC83}, \cite{aC85}  and \cite{mW94}.) 

This was achieved by K. Dykema, T. Figiel, G. Weiss and M. Wodzicki 
\cite{DFWW} and \cite{kDgWmW00} who fully characterized commutator spaces in terms of arithmetic (Cesaro) means of monotone sequences 
\cite[Theorem 5.6]{DFWW}
thus concluding a line of research introduced in G. Weiss' PhD thesis \cite{gW75} (see also \cite{gW80} and \cite{gW86})
and developed significantly by 
N. Kalton in \cite{nK89}. 

The introduction of arithmetic mean operations on ideals and the results in \cite{DFWW} in particular, 
opened up a new area of investigation in the study of operator ideals and have become an intrinsic part of the theory.
To explore this area is the goal of our program outlined in \cite{vKgW02}, of which this paper and \cite{vKgW07-soft} are the beginning. 
In this paper we focus mainly on the question: ``How many nonzero traces can an ideal support?" and on developing tools to investigate it.

From \cite{DFWW} we know that an ideal supports ``no nonzero traces" precisely when the ideal is stable under the arithmetic mean (am-stable). 

In Section 6 we prove that an ideal that does not contain the diagonal operator $diag <1,\frac{1}{2},\frac{1}{3},\dots>$ supports ``one" nonzero trace  precisely when the ideal is stable under the arithmetic mean at infinity (am-$\infty$ stable). Here, what is meant by ``one" is that the ideal supports a trace that is unique up to scalar multiples. 

In Section 7 we prove that ``infinitely many" traces are supported by ideals whose soft-interior or soft-complement are not am-stable or not am-$\infty$ stable and by other classes of ideals as well. 
This motivates our conjecture that the number of traces that an ideal can support must always be either ``none", ``one", or ``infinitely many". 

In the first part of this paper we develop the above mentioned notions of arithmetic mean, arithmetic mean at infinity, soft interior and soft complement of ideals as their interplay provides the tools for this study.

The arithmetic mean of operator ideals was introduced and played an important role in \cite{DFWW}; we review some if its properties in Section 2. 

While the soft-interior and soft-complement of ideals have appeared implicitly in numerous situations in the literature, to the best of our knowledge they have never been formally studied before. 
We introduce them briefly in Section 3 and study their interplay with the am operations; 
we leave to \cite{vKgW07-soft} a more complete development of these notions and of the ensuing ideal classes. 

The arithmetic mean at infinity was used among others in \cite{AGPS87}, \cite{DFWW}, and \cite{mW02} as an operation on sequences. In Section 4 we develop the properties of the am-$\infty$ operations on ideals which parallel only in part those of the am operations and we study their interplay with the soft-interior and soft-complement operations. 
The notion of regularity for sequences, which figured prominently in the study of principal ideals in \cite{GK69} and was essential for the study of positive traces on principal ideals in \cite{jV89}, has a dual form for summable sequences that we call regularity at infinity (Definition \ref{D:am-infty stability/regularity}). 
In Theorem \ref{T: Theorem 4.12} we link regularity at infinity to other sequence properties, including a Potter type inequality used by Kalton in \cite{nK87} and Varga type properties (cf. \cite{jV89}) and to the Matuszewska index introduced in this context in \cite{DFWW}. 

In Section 5, we study trace extensions from one ideal to another and in the process we obtain hereditariness (solidity) of the cone of positive operators 
$(\mathscr L_1+ [I,B(H)])^+$ where $\mathscr L_1$ is the trace class. 
(The analogous result for $(F+[I, B(H)])^+$,  where $F$ is the finite rank ideal, is obtained in Corollary \ref{C:(F+[])+}.) 
These results are applied in Propositions \ref{P:FugL2} and \ref{P:Shulman problem} to two problems on elementary operators studied by V. Shulman (private communications related to \cite{vS83}). 

We do not know for which ideals, if not all, the cones $(J+[I,B(H)])^+$ are hereditary beyond the cases 
$J = \{0\}, F, \mathscr L_1$, $I \subset J$ or $J \subset [I,B(H)]$.

In Section 6 we characterize those ideals of trace class operators that support a unique trace (up to scalar multiples): 
 they are precisely the am-$\infty$ stable ideals (Theorem \ref{T:i-iii}). 

In Section 7 we bring the previously developed tools to bear on the question of how many traces an ideal can support.  

Ideals divide naturally into three classes from the perspective developed here: 

$\bullet$   the ``small" ideals, i.e., the ideals contained in the largest am-$\infty$ stable ideal $st_{a_\infty} (\mathscr L_1) \subset \mathscr L_1$, 
(see Definition 4.14) 

$\bullet$   the ``large" ideals, i.e., the ideals containing the smallest am-stable ideal $st^a(\mathscr L_1)$, (see ibid) 

$\bullet$   the ``intermediate" ideals, i.e., all remaining ideals. 

Intermediate ideals always support infinitely many traces, or more precisely, $[I,B(H)]$ has uncountable codimension in $I$ 
(Theorem \ref{T: stabilizer containments and uncountable dimension}(iii)). 

We conjecture that the codimension of $[I,B(H)]$ in $I$ can be only one or infinity for small ideals and zero or infinity for large ideals. 

The conjecture is proven for soft-edged ideals ($I=seI$) and soft-complemented ideals ($I = scI$) (Corollary \ref{C: soft-edged or soft-complemented}). 
These include all the classical ideals we examined including all those investigated in \cite{DFWW}. 

A stronger result (Theorem \ref{T: stabilizer containments and uncountable dimension}) is that if $seI$ (or, equivalently, $scI$) is not am-$\infty$ stable (for small ideals) or am-stable (for large ideals), then $[I,B(H)]$ has uncountable codimension in $I$. 

This leaves the conjecture open for small ideals that are not am-$\infty$ stable and for large ideals that are not am-stable but have am-stable soft interiors. 

The key technical tool for these results is Theorem \ref{T: BASICS: seJ notin F+[I,B(H)] implies uncountable dimension-(J+[I,B(H)])/(F+[I,B(H)])} which states that $[I,B(H)]$ has uncountable codimension in $I$ whenever $seI$ is not contained in $F+[I,B(H)]$. 

In Theorem \ref{T: false converse for T: se or sc inf codim implies} and Corollary \ref{C: false converse for T: se or sc inf codim implies+} a different technique shows that $I$ supports infinitely many traces for a class of ideals that include some cases where $seI$ is am-stable. 

Following this paper (the first of the program outlined in \cite{vKgW02}) is \cite{vKgW07-soft} where we study the soft-interior and soft-complement operations on ideals and their interplay with the am and am-$\infty$
operations. In forthcoming papers we will investigate: 

(1) Connections between (infinite) majorization theory, stochastic matrices, infinite convexity notions for ideals, diagonal invariance, 
and the am and am-$\infty$ operations \cite {vKgW07-Majorization}. 

(2) Lattice structure for $B(H)$ and for some distinguished classes of ideals and their density properties. 
Example: between two distinct principal ideals, at least one of which is am-stable (resp., am-$\infty$ stable), lies a third am-stable (resp., am-$\infty$ stable) principal ideal \cite{vKgW07-Density}. 

(3) First and second order arithmetic mean cancellation and inclusion properties \cite {vKgW07-Density}. 
Example: for which ideals $I$ does $I_a=J_a$ (resp., $I_a\subset J_a$, $I_a\supset J_a$) imply $I=J$ (resp., $I\subset J$, $I\supset J$)?
Are there ``optimal" ideals $J$ for the inclusions $I_a\subset J_a$ and $I_a\supset J_a$? 
In the principal ideal case concrete answers are obtained. 
For instance, $(\xi)_a = (\eta)_a$ implies $(\xi)=(\eta)$ for every $\eta$ if and only if $\xi$ is regular.

(4) In \cite {vKgW07-2nd} conditions on $\xi$ are given which guarantee that $(\xi)_{a^2} = (\eta)_{a^2}$ implies $(\xi)_a = (\eta)_a$ and a counterexample to this implication for the general case is provided, which settles a question of Wodzicki.

\section{\leftline {\bf Preliminaries and the arithmetic mean}}

The natural domain of the usual trace $Tr$ on $B(H)$ (with $H$ a separable infinite-dimensional complex Hilbert space) is the trace class ideal $\mathscr{L}_1$. 
However, ideals of $B(H)$ can support other traces.

\begin{definition}\label{D:trace} 
A trace $\tau$ on an ideal $I$ is a unitarily invariant linear functional on $I$. \linebreak
\end{definition}
\noindent In this paper, traces are neither assumed to be positive nor faithful. \textit{All ideals are assumed to be proper.}

Since $UXU^* - X \in [I,B(H)]$ for every $X \in I$ and every unitary operator U and since unitary operators span $B(H)$, unitarily invariant linear functionals on an ideal $I$ are precisely the linear functionals on $I$ that vanish on the \emph{commutator space}  $[I,B(H)]$. Also known as the \emph{commutator ideal} it is defined as the linear span of 
commutators of operators in $I$ with operators in $B(H)$. 
Thus traces can be identified with the elements of the linear dual of the quotient space $\frac{I}{[I,B(H)]}$  and hence to $\frac{I}{[I,B(H)]}$ itself. 

A constant theme in the theory of operator ideals has been its connection to the theory of sequence spaces.  

Calkin \cite{jC41} established a correspondence between the two-sided ideals of $B(H)$
and the \emph{characteristic sets}, i.e., the positive cones of $\text{c}_{\text{o}}^*$ 
(the collection of sequences decreasing to $0$) that are hereditary and invariant under \emph{ampliations} 
\[
\text{c}_{\text{o}}^* \owns \xi \rightarrow D_m\xi:=~<\xi_1,\dots,\xi_1,\xi_2,\dots,\xi_2,\xi_3,\dots,\xi_3,\dots>
\] 
where each entry $\xi_i$ of $\xi$ is repeated $m$-times. 
The order-preserving lattice isomorphism 
$I \rightarrow \Sigma(I)$ maps each ideal to its characteristic set $\Sigma(I) := \{s(X) \mid X \in I\}$ where $s(X)$ denotes the sequence of 
$s$-numbers of $X$, i.e., all the eigenvalues of $|X| = (X^*X)^{1/2}$ repeated according to multiplicity, arranged in decreasing order, and completed by adding infinitely many zeroes if $X$ has finite rank.  
Conversely, for every characteristic set $\Sigma \subset \text{c}_{\text{o}}^*$, 
if $I$ is the ideal generated by $\{diag~\xi \mid \xi \in \Sigma\}$ 
where $diag~\xi$ is the 
diagonal matrix with entries $\xi_1,\xi_2,\dots$ , then we have $\Sigma = \Sigma(I)$.  
We shall also need the sequence space $S(I) := \{\xi \in \text{c}_{\text{o}} \mid |\xi|^* \in \Sigma(I)\}$ where $|\xi|^*$ denotes the monotonization (in decreasing order) of $|\xi|$.
Equivalently, $S(I) = \{\xi \in \text{c}_{\text{o}} \mid diag\,\xi \in I\}$.

More recently Dykema, Figiel, Weiss, and Wodzicki \cite{DFWW} characterized the normal operators in the commutator spaces $[I,B(H)]$ 
in terms of spectral sequences. 
An important feature is that membership in commutator spaces (noncommutative objects) is reduced to certain conditions on associated sequences (commutative ones). 

When $X \in K(H)$, the ideal of compact 
operators on $H$, denote an ordered spectral sequence for $X$ by $\lambda(X) :=~ <\lambda(X)_1,\lambda(X)_2,\dots>$, 
i.e., a sequence of all the eigenvalues of $X$ (if any), 
repeated according to algebraic multiplicity, 
completed by adding 
infinitely many zeroes when only finitely many eigenvalues are nonzero, and arranged in any order so that $|\lambda(X)|$ is nonincreasing.  For any sequence $\lambda =~ <\lambda_n>$, denote by $\lambda_a$ the sequence of its arithmetic (Cesaro) mean, i.e.,
\[
\lambda_a :=~ <\frac{1}{n}\sum_{j=1}^{n}\lambda_j>_{n=1}^{\infty}. 
\]

\noindent A special case of \cite[Theorem 5.6]{DFWW} (see also \cite[Introduction]{DFWW}) is:

\begin{theorem} \label{T:DFWW}
Let $I$ be a proper ideal, let $X \in I$ be a normal operator, and let $\lambda(X)$ be any ordered spectral sequence for $X$.
Then $X \in [I,B(H)]$ if and only if $\lambda(X)_a \in S(I)$ 
if and only if $|\lambda(X)_a| \leq \xi$ for some $\xi \in \Sigma(I)$.
\end{theorem}

\noindent
In fact, the conclusion holds under the less restrictive condition that $\lambda(X)$ is ordered so that $|\lambda(X)| \leq \eta$ for some $\eta \in \Sigma(I)$ 
(see \cite[Theorem 5.6]{DFWW}).

Arithmetic means first entered the analysis of $[\mathscr L_1,B(H)]$ for a special case in \cite{gW75} and \cite{gW80} and for 
its full characterization in \cite{nK89}.  
As the main result in \cite{DFWW} (Theorem 5.6) conclusively shows, arithmetic means are essential for the study of traces and commutator spaces in operator ideals. 
\cite{DFWW} also initiated a systematic study of ideals derived via arithmetic mean operations (am ideals for short). For the reader's convenience we list the definitions and first properties from \cite[Sections 2.8 and 4.3]{DFWW}.

If $I$ is an ideal, then the arithmetic mean ideals $_aI$ and $I_a$, called respectively the \emph{pre-arithmetic mean} and \emph{arithmetic mean} of $I$, are the ideals with characteristic sets
\[
\Sigma(_aI) := \{\xi \in \text{c}_{\text{o}}^* \mid \xi_a \in \Sigma(I)\} 
\]
\[
\Sigma(I_a) := \{\xi \in \text{c}_{\text{o}}^* \mid \xi = O(\eta_a)~\text{for some}~ \eta \in \Sigma(I)\}.
\]
The \emph{arithmetic mean-closure} $I^-$ and the \emph{arithmetic mean-interior} $I^o$ of an ideal 
(am-closure and am-interior for short) are defined as 
\[
I^- :=~ _a(I_a) \qquad \text{and} \qquad I^o := (_aI)_a
\]
and for any ideal $I$, the following \emph{5-chain of inclusions} holds: 
\[
_aI \subset I^o \subset I \subset I^- \subset I_a.
\]

A restriction of Theorem \ref{T:DFWW} to positive operators can be reformulated in terms of pre-arithmetic means as:
\[
[I,B(H)]^+ = (_aI)^+
\]
where $(_aI)^+$ denotes the cone of positive operators in $_aI$.
As a consequence, $[I,B(H)]^+$ is always hereditary, i.e., solid, $_aI \subset [I,B(H)] \subset I$, and $_aI$ and $I$ are, 
respectively, the largest ideal contained in $[I,B(H)]$, and
the smallest ideal containing $[I,B(H)]$ (for the latter see Remark \ref{R: X,|X|,X oplus -X,(X)}(c)). 
Also, $I=[I,B(H)]$ if and only if $I=~_aI$.  
An ideal with the latter property is called \emph{arithmetically mean stable} (am-stable for short) and it is easy to see, using the 5-chain mentioned above, that a necessary and sufficient condition for am-stability is that $I=I_a$.  
Am-stability for many classical ideals and powers of ideals was studied extensively in \cite[5.13-5.27]{DFWW}. 

For $\xi \in \text{c}_{\text{o}}^*$, denote by $(\xi)$ the principal ideal generated by $diag~\xi$.  
Notice that if $\xi,\eta \in \text{c}_{\text{o}}^*$, then $(\xi) \subset (\eta)$ if and only if $\xi = O(D_m\eta)$ for some $m \in \mathbb N$; 
so $(\xi) = (\eta)$ if and only if both $\xi = O(D_m\eta)$ and $\eta = O(D_k\xi)$ hold for some $m,k \in \mathbb N$.  
Thus $(\xi) = (\eta)$ implies $\xi \asymp \eta$ (i.e., $\xi = O(\eta)$ and $\eta = O(\xi)$) 
if and only if $\xi$ (and hence $\eta$) satisfies 
the $\Delta_{1/2}$-condition,
i.e., $\xi \asymp D_2\xi$ (which is equivalent to $\xi \asymp D_m\xi$ for all $m\in \mathbb N$).  
In this context recall the well-known $\Delta_2$-condition for nondecreasing sequences $\underset{n}{\sup}\frac{g_{2n}}{g_n}<\infty$,
i.e., $g \asymp~ D_{1/2}\,g$ where $(D_{1/m}g)_n := g_{mn}$. 
 
The arithmetic mean $\xi_a$ of a sequence $\xi \in \text{c}_{\text{o}}^*$ always satisfies the elementary inequality $D_2\xi_a \leq 2\xi_a$ and hence also the
$\Delta_{1/2}$-condition. 
From this  
it follows easily that $(\xi_a)=(\xi)_a$ and hence that the principal ideal $(\xi)$ is am-stable if and only if $\xi \asymp \xi_a$, i.e., 
$\xi$ is regular (cf. \cite[p.143 (14.12)]{GK69}). 
The notion of regularity plays a crucial role in Varga's study of positive traces on principal ideals \cite{jV89}.

Of special importance in \cite{DFWW} and in this paper is the principal ideal $(\omega)$, where $\omega$ denotes the harmonic sequence $<\frac{1}{n}>$.  
Elementary computations show that $F_a = (\mathscr L_1)_a = (\omega)$ and that $_a(\omega) = \mathscr L_1$. 
Hence $_aI \ne \{0\}$ if and only if $\omega \in \Sigma(I)$. 
An immediate but important consequence of Theorem \ref{T:DFWW} which will be used often throughout this paper is that 
$\omega \in \Sigma(I)$ if and only if $\mathscr L_1 \subset [I,B(H)]$ 
if and only if $F \subset [I,B(H)]$.

\section{\leftline {\bf Soft interior and soft cover of ideals}}

As mentioned in the Introduction, Theorem \ref{T: BASICS: seJ notin F+[I,B(H)] implies uncountable dimension-(J+[I,B(H)])/(F+[I,B(H)])}, 
which is one of our main results, is formulated in terms of the notion of the soft interior of an ideal.

\begin{definition}\label{D: se,sc}
Given an ideal $I$, the soft interior of $I$ is the ideal $se I := IK(H)$ with characteristic set 
\[
\Sigma(seI) := \{\, \xi \in \text{c}_{\text{o}}^* \mid \xi \leq \alpha \eta~\mbox{for some}~\alpha \in \text{c}_{\text{o}}^*,~ \eta \in \Sigma(I) \}.
\] 
The soft cover of $I$ is the ideal $sc I$ with characteristic set 
\[
\Sigma(sc I) := \{\, \xi \in \text{c}_{\text{o}}^* \mid \alpha \xi \in \Sigma(I) \mbox{ for all } \alpha \in \text{c}_{\text{o}}^* \}.
\]

An ideal $I$ is called soft-edged if $se I = I$ and it is called soft-complemented if $sc I = I$.
A pair of ideals $I \subset J$ is called a soft pair if $I=se J$ and $sc I = J$.
\end{definition}

It is immediate to verify that the sets $\Sigma(se I)$ and $\Sigma(sc I)$ are indeed characteristic sets and that 
in the notations of \cite[Section 2.8]{DFWW}, $sc I := I/K(H)$.
Notice that $se I$ is the largest soft-edged ideal contained in $I$ and $sc I$ is the smallest soft-complemented ideal containing $I$.
Also needed in this paper and easy to show is that, for every ideal $I$, 
$sc~seI=scI$, $se~scI=seI$, $seI \subset I \subset scI$, and $seI \subset scI$ is a soft pair. (cf. \cite{vKgW02},\cite{vKgW07-soft}).

\begin{remark}\label{R:soft terminology}
This terminology is motivated by the fact that $I$ is soft-edged if and only if for every 
$\eta \in \Sigma(I)$ there is some $\xi \in \Sigma(I)$ such that \ $\eta = o(\xi)$. 
Similarly, $I$ is soft-complemented if and only if, for every 
$\xi \in \text{c}_{\text{o}}^* \setminus \Sigma(I)$, 
there is some $\eta \in \text{c}_{\text{o}}^* \setminus \Sigma(I)$ such that $\eta = o(\xi)$.
\end{remark}

Soft-edged and soft-complemented ideals and soft pairs are common 
among the classical operator ideals, that is, ideals $I$ whose $S(I)$-sequence spaces are classical sequence spaces.
In \cite{vKgW07-soft} we show that the following are soft-complemented: 
countably generated ideals, the normed ideals $\mathfrak S_{\phi}$ induced by a symmetric norming function $\phi$, 
Orlicz ideals $\mathscr L_M$, 
Lorentz ideals generated by a nondecreasing $\Delta_2$-function and, more generally, 
ideals whose characteristic set is a quotient of the characteristic set of a soft-complemented ideal by an arbitrary set of sequences 
$X \subset [0,\infty)^{\mathbb Z^+}$
and hence, in particular, K\"othe duals $\mathscr L_1/X$ and quotients $I/J$ of a soft-complemented ideal 
by an arbitrary ideal (see \cite[Section 2.8]{DFWW} for a discussion on quotients). 
Also the following are soft-edged: 
the ideals $\mathfrak S_{\phi}^{\mbox{(o)}}$ (the closure of $F$ under the norm of $\mathfrak S_{\phi}$), 
small Orlicz ideals $\mathscr L_M^{\mbox{(o)}}$, and Lorentz ideals generated by a nondecreasing $\Delta_2$-function. 
Moreover, $\mathfrak S_{\phi}^{\mbox{(o)}} \subset \mathfrak S_{\phi}$ and 
$\mathscr L_M^{\mbox{(o)}} \subset \mathscr L_M$ 
are natural examples of soft pairs. 
(See \cite{DFWW} for a convenient reference for these classical ideals.) 

The condition $seI \not \subset F+[I,B(H)]$ in  
Theorem \ref{T: BASICS: seJ notin F+[I,B(H)] implies uncountable dimension-(J+[I,B(H)])/(F+[I,B(H)])}
will need to be reformulated in terms of arithmetic means and arithmetic means at infinity.
The first step is established by the following commutation relations between the arithmetic and pre-arithmetic mean ideal operations and the soft interior and soft complement operations.

\begin{lemma}\label{L:se,sc,pre-am,am}
Let $I$ be an ideal. Then
\item[(i)] $sc(_aI) \subset~ _a(scI)$
\item[(i$'$)] $sc(_aI) =~ _a(scI)$ if and only if 
$\omega \notin \Sigma(scI) \setminus \Sigma(I)$.
\item[(ii)] $se(I_a) \subset (seI)_a$
\item[(ii$'$)] $se(I_a) = (seI)_a$ if and only if 
either $I \not\subset \mathscr L_1$ or $I = \{0\}$.
\end{lemma}

\begin{proof}
\item[(i)] For $\xi \in \Sigma(sc(_aI))$ and all $\alpha \in \text{c}_{\text{o}}^*$, 
by the definition of soft complement, $\alpha\xi \in \Sigma(_aI)$, that is, $(\alpha \xi)_a \in \Sigma(I)$. 
But clearly $\alpha\xi_a \leq (\alpha\xi)_a$ and hence $\alpha\xi_a \in \Sigma(I)$. 
Thus $\xi_a \in \Sigma(scI)$ and $\xi \in \Sigma(_a(scI))$.

\item[(i$'$)]
Recall from the end of Section 2 that for any ideal $J$, $_aJ = \{0\}$ if and only if $\omega \notin \Sigma(J)$. 
Consider separately the three cases: 
$\omega \notin \Sigma(scI)$, $\omega \in \Sigma(scI) \setminus \Sigma(I)$ and $\omega \in \Sigma(I)$.
If $\omega \notin \Sigma(scI)$, then $_a(scI) = \{0\}$ and equality holds. 
If $\omega \in \Sigma(scI) \setminus \Sigma(I)$, then $_a(scI) \ne \{0\}$ 
but $_aI = \{0\}$, so $sc(_aI) = \{0\}$ and equality fails. 
Finally assume that $\omega \in \Sigma(I)$ and let 
$\xi \in \Sigma(_a(scI))$ and $\alpha \in \text{c}_{\text{o}}^*$.
In case $\xi \in \ell^1$ then $\alpha\xi \in \ell^1$, hence $(\alpha\xi)_a = O(\omega)$, 
$\alpha\xi \in \Sigma(_aI)$ and thus $\xi \in \Sigma(sc(_aI))$, so equality holds.
In case $\xi \notin \ell^1$, it is easy to verify that  
$\widetilde \alpha_n := (\frac{(\alpha\xi)_a}{\xi_a})_n = \frac{\sum_1^n \alpha_j\xi_j}{\sum_1^n \xi_j}  \downarrow 0$ 
and hence $(\alpha\xi)_a = \widetilde \alpha \xi_a \in \Sigma(I)$. 
Thus $\alpha \xi \in \Sigma(_aI)$ and hence 
$\xi \in \Sigma(sc(_aI))$ and equality holds.

\item[(ii)] If $\xi \in \Sigma(se(I_a))$ 
then $\xi \leq \alpha \eta_a$ for some $\eta \in \Sigma(I)$ and $\alpha \in \text{c}_{\text{o}}^*$. 
Then from the inequality in the proof of (i), $\xi \leq (\alpha \eta)_a \in \Sigma((seI)_a)$.

\item[(ii$'$)] Consider separately the three cases: 
$I = \{0\}$, $\{0\} \ne I \subset \mathscr L_1$ and $I \not\subset \mathscr L_1$.
In the case $I = \{0\}$ the equality is trivial, and if $\{0\} \ne I \subset \mathscr L_1$, 
it fails trivially since $(seI)_a = I_a = (\omega)$ 
(recalling that $F_a = (\mathscr L_1)_a = (\omega)$) 
and since $(\omega)$ is not soft-edged. 
In the case that $I \not\subset \mathscr L_1$,
for each 
$\xi \in \Sigma((seI)_a)$, 
$\xi \leq \rho_a$ for some $\rho \in \Sigma(seI)$, i.e., $\rho \leq \alpha\eta$ for some $\eta \in \Sigma(I)$ and $\alpha \in \text{c}_{\text{o}}^*$. 
By adding to $\eta$ if necessary an element of $\Sigma(I) \setminus \ell^1$, one can insure that $\eta \notin \ell^1$, 
in which case, from the proof of (i$'$), again there is an $\widetilde \alpha \in \text{c}_{\text{o}}^*$ for which 
$(\alpha\eta)_a = \widetilde \alpha \eta_a$. 
But then $\xi \leq \widetilde \alpha \eta_a \in \Sigma((se(I_a))$. By (ii) equality holds.
\end{proof}
\pagebreak
\begin{proposition}\label{P:TFAE-seI am-stable, scI am-stable, seI in Ipre-a, Ia in scI} 
Let $I$ be an ideal. Then the following are equivalent.
\item[(i)] $seI$ is am-stable
\item[(ii)] $scI$ is am-stable
\item[(iii)] $seI \subset ~_aI$
\item[(iii$'$)] $seI  \subset [I, B(H)]$
\item[(iv)] $scI \supset I_a$
\end{proposition}

\begin{proof}
Assume first that $I \not\subset \mathscr L_1$. 
\item[(i) $\Rightarrow$ (iii)] Since $seI \subset I$ and the pre-arithmetic mean is inclusion 
preserving, \linebreak
$seI = ~_a(seI) \subset~_aI$. 
\item[(iii) $\Leftrightarrow$ (iii$'$)] Obvious since $[I, B(H)]^+ = ~_aI^+$ by Theorem \ref{T:DFWW}.
\item[(iii) $\Rightarrow$ (ii)] 
Condition (ii) is immediate from the chain of relations 
\begin{center}$scI = sc~seI \subset sc(_aI) \subset~ _a(scI) \subset scI$\end{center}
which implies equality. 
For the first equality, recall the paragraph following Definition \ref{D: se,sc}; 
$sc$ being inclusion preserving implies the first inclusion; 
Lemma \ref{L:se,sc,pre-am,am}(i) implies the second inclusion; and the 5-chain of inclusions implies the last inclusion. 
\item[(ii) $\Rightarrow$ (iv)] From $I \subset scI$ it follows that $I_a \subset (scI)_a = scI$. 
\item[(iv) $\Rightarrow$ (i)] In case $I \not\subset \mathscr L_1$,  
$(seI)_a = se(I_a) \subset se~scI = seI \subset (seI)_a$,
where the first equality follows from Lemma \ref{L:se,sc,pre-am,am}(ii$'$), the first inclusion holds since se is inclusion preserving,
the second equality holds for all ideals (recall again the paragraph following Definition \ref{D: se,sc}) and the last inclusion follows from the 5-chain.

In case $I \subset \mathscr L_1$, if $I = \{0\}$ then all four conditions are trivially true and  
if \linebreak
$I \ne \{0\}$ then all four are false. 
Indeed the arithmetic mean of any nonzero ideal, and hence every nonzero am-stable ideal, 
must contain $(\omega)$ while $I \subset \mathscr L_1$ implies that $seI \subset scI \subset \mathscr L_1 \subsetneq (\omega)$, which shows that (i), (ii), and (iv) are false. 
And since $_aI \subset~_a(\mathscr L_1) = \{0\}$ (recall the last paragraph of Section 2) but $seI\ne \{0\}$, (iii) too is false.
\end{proof}

\begin{remark}\label{R: I am-stable implies seI,scI am-stable with false converse}
The $se$ and $sc$ operations preserve am-stability by Lemma \ref{L:se,sc,pre-am,am}(i) and Proposition \ref{P:TFAE-seI am-stable, scI am-stable, seI in Ipre-a, Ia in scI}.
But am-stability of $seI$ (or $scI$) does not imply am-stability of $I$ as shown by the construction in Theorem \ref{T: false converse for T: se or sc inf codim implies}.
\end{remark}

\section{\leftline {\bf Arithmetic mean at infinity}}

Since $\xi_a \asymp \omega$ for every $0 \neq \xi \in (\ell^1)^* := \ell^1 \cap \text{c}_{\text{o}}^*$, 
nonzero ideals $I \subset \mathscr L_1$ all have arithmetic mean $I_a=(\omega)$.
Thus the (Cesaro) arithmetic mean is not adequate for distinguishing between ideals contained in $\mathscr L_1$.  
For such ideals, one needs to employ instead the \emph{arithmetic mean at infinity}
\[
\xi_{a_\infty} :=~ <\frac{1}{n}\sum_{n+1}^{\infty} \xi_j>
\]
(see \cite[Section 2.1 (16)]{DFWW},\cite{nK87},\cite{mW02}).

In this section we develop properties of the am-$\infty$ operation on sequences including a characterization of $\infty$-regular sequences which is dual to the known characterization of regular sequences and we introduce and investigate the am-$\infty$ operations on ideals. 
This will lead us to Proposition \ref{P: TFAE:seI am-stable, scI am-stable, seI in, scI contains}, which we find essential for Section 7.

The following lemma analyzes the relations between the am-$\infty$ operation and the $D_m$ operations on sequences.
Recall that if $j=mn-p$ with $n,p$ integers, $n\geq 1$, and $0 \leq p \leq m-1$, then $(D_m\xi)_j=\xi_n$.

\begin{lemma}\label{L:relations between Dmxi,(Dmxi)ainfty,Dmxiainfty}  
Let $\xi \in (\ell^1)^*$. Then for $m=2,3,\dots$ one has 
\item[(i)] $D_m\xi_{a_\infty} \leq (D_m\xi)_{a_\infty}$
\item[(ii)] $(D_m\xi_{a_\infty})_j \geq ((D_{m-1}\xi)_{a_\infty})_j$ when $j\geq(m-1)(m-2)$
\item[(iii)] $(D_m\xi_{a_\infty})_j \geq \frac{1}{2(m-1)}(D_{m-1}\xi)_j$ when $j\geq2m(m-1)$
\item[(iv)]  $(D_{1/m}\xi)_j := \xi_{mj} \leq \frac{1}{m-1}(\xi_{a_\infty})_j$ for all $j$.
\end{lemma}

\begin{proof}
\item[(i)]   Let $j=mn-p$ with $n,p$
integers , $n \geq 1$, and $0 \leq p \leq m-1$.  Then 
\begin{align*}
((D_m~\xi)_{a_\infty})_j &= \frac{1}{mn-p} \sum_{mn-p+1}^{\infty} (D_m~\xi)_i 
= \frac{1}{mn-p} (p~\xi_n + m\sum_{n+1}^{\infty} \xi_i) \\
&= \frac{p}{mn-p}~ \xi_n + \frac{mn}{mn-p} (\xi_{a_\infty})_n \\
&= \frac{p}{mn-p}~ \xi_n + \frac{mn}{mn-p} (D_m~\xi_{a_\infty})_j \geq (D_m~\xi_{a_\infty})_j.
\end{align*}

\item[(ii)] Let $j=mn-p$ as above.  Then 
\begin{align*}
(D_m~\xi_{a_\infty})_j = (\xi_{a_\infty})_n = (D_{m-1}\xi_{a_\infty})_{(m-1)n} &= ((D_{m-1}\xi)_{a_\infty})_{(m-1)n} \\
&\geq ((D_{m-1}\xi)_{a_\infty})_j.
\end{align*}
The third equality follows from the proof of (i) for the case $p=0$, and the inequality holds since $(D_{m-1}\xi)_{a_\infty}$ is nonincreasing and 
$j \geq (m-1)(m-2)$ implies 
$j \geq (m-1)n$ by elementary calculation. 
\item[(iii)] Let $2m(m-1) \leq j=mn-p = (m-1)(n+k)-p'$ where 
$0 \leq p \leq m-1$ and $0 \leq p' \leq m-2$, and hence $k \geq 1$.  
Then 
\begin{align*}
(D_m~\xi_{a_\infty})_j &= (\xi_{a_\infty})_n = \frac{1}{n} \sum_{n+1}^{\infty} \xi_j \geq \frac{1}{n} \sum_{n+1}^{n+k} \xi_j \\
&\geq \frac{k}{n}\xi_{n+k} = \frac{k}{n}(D_{m-1}\xi)_j \geq \frac{1}{2(m-1)}(D_{m-1}\xi)_j
\end{align*}
where the latter inequality holds since $mn \geq j\geq2m(m-1)$ and hence
\[
\frac{k}{n} = \frac{1}{m-1}(1+\frac{p'-p}{n}) \geq \frac{1}{m-1}(1-\frac{m-1}{n}) \geq \frac{1}{2(m-1)}.
\]
\item[(iv)] Immediate by the monotonicity of $\xi$ since then
\[
(m - 1)j\xi_{mj} \leq \sum_{i=j+1}^{mj} \xi_i \leq j(\xi_{a_\infty})_j.
\]
\end{proof}
\pagebreak
\begin{remark}\label{R:sharpness for relations between Dmxi,(Dmxi)ainfty,Dmxiainfty}
It is easy to see that the bounds in (i),(ii), and (iv) are sharp.
In lieu of the bound $\frac{1}{2(m-1)}$ in (iii) we can obtain $\frac{1-\epsilon}{m-1}$ for any $\epsilon > 0$, 
but not for $\epsilon = 0$, i.e., $D_m\xi_{a_\infty}$ does not majorize $\frac{1}{m-1}D_{m-1}\xi$, even for $j$ large enough.   
Indeed, for any j, set $\xi_i = 1$ for $1\leq i \leq 2j-1$ and $0$ elsewhere. 
Then 
\begin{center}$(D_2~\xi_{a_\infty})_{2j-1} = (\xi_{a_\infty})_j = \frac{j-1}{j}<1=\xi_{2j-1}$.\end{center}
\end{remark}

\begin{corollary}\label{C:(xi)subset(xi ainfty)}
If $\xi \in (\ell^1)^*$ then $(\xi) \subset (\xi_{a_\infty})$.
\end{corollary}

\begin{proof}
By Lemma \ref{L:relations between Dmxi,(Dmxi)ainfty,Dmxiainfty} (iii) for $m=2$ we have  
$\xi_j \leq 2(D_2\xi_{a_\infty})_j$ for $j \geq 4$,\, and thus $\xi \in \Sigma((\xi_{a_\infty}))$.
\end{proof}

In contrast to the arithmetic mean case where the sequence $\xi_a$ always satisfies the $\Delta_{1/2}$-condition, 
Example \ref{E:infty regulars/irregulars}(ii) below shows that this is not always true for $\xi_{a_\infty}$. 
Moreover, Example \ref{E:infty regulars/irregulars}(iii) shows that $\xi_{a_\infty}$ may satisfy the $\Delta_{1/2}$-condition while $\xi$ does not.
Corollary \ref{C:(xi) subset (xiainfty)}(ii) provides a necessary and sufficient condition for $\xi_{a_\infty}$ to satisfy the $\Delta_{1/2}$-condition.

\begin{corollary}\label{C:(xi) subset (xiainfty)}
Let $\xi \in (\ell^1)^*$. 
\item[(i)] If $\xi$ satisfies the $\Delta_{1/2}$-condition, so does $\xi_{a_\infty}$.
\item[(ii)] $\xi_{a_\infty}$ satisfies the $\Delta_{1/2}$-condition if and only if $\xi = O(\xi_{a_\infty})$.
\end{corollary}

\begin{proof}
\item[(i)]  If $D_2\xi \leq M\xi$ for some $M>0$, then by Lemma \ref{L:relations between Dmxi,(Dmxi)ainfty,Dmxiainfty}(i), 
$D_2\xi_{a_\infty} \leq (D_2\xi)_{a_\infty} \leq  M\xi_{a_\infty}$.
\item[(ii)] If $\xi_{a_\infty}$ satisfies the $\Delta_{1/2}$-condition then $\xi = O(D_2 \xi_{a_\infty}) = O(\xi_{a_\infty})$ by Lemma \ref{L:relations between Dmxi,(Dmxi)ainfty,Dmxiainfty}(iii). 
Conversely, assume that $0\ne \xi=O(\xi_{a_\infty})$, i.e., 
$\xi_k \leq \frac{M}{k}\sum_{j=k+1}^\infty \xi_j$ for some $M>0$ and all $k$. 
Then $\xi_k > 0$ for all $k$ and 
\begin{align*}
\frac{\sum_{j=n+1}^\infty \xi_j}{\sum_{j=2n+1}^\infty \xi_j} &= \prod_{k=n+1}^{2n} (1+\frac{\xi_k}{\sum_{j=k+1}^\infty \xi_j}) 
\leq \prod_{k=n+1}^{2n} (1+\frac{M}{k}) \\
&= e^{\sum_{k=n+1}^{2n} log\,(1+\frac{M}{k})} \leq e^{M \sum_{k=n+1}^{2n} \frac{1}{k}}  \leq 2^M.
\end{align*}
Hence $\frac{(\xi_{a_\infty})_n}{(\xi_{a_\infty})_{2n}} \leq 2^{M+1}$ for all $n$, i.e., $\xi_{a_\infty}$ satisfies the $\Delta_{1/2}$-condition.
\end{proof}
\begin{example}\label{E:infty regulars/irregulars}
\item[(i)] Let $\xi = \omega^p$ where $p>1$. Then $\xi_{a_\infty} \asymp \xi$ satisfies the $\Delta_{1/2}$-condition.
\item[(ii)]Let $\xi =~ <q^n>$ where $0 < q < 1$. Then $\xi_{a_\infty} = o(\xi)$ and neither sequence satisfies the $\Delta_{1/2}$-condition.
\item[(iii)] Let $n_k$ be an increasing sequence of integers for which $n_k \ge kn_{k-1}$ (with $n_1=1$ and $k \ge 2$), 
let $<\epsilon_k> ~\in \text{c}_{\text{o}}^*$ where $\sum_{k=1}^{\infty} \epsilon_k n_k < \infty$,
and let $\xi_n := \epsilon_k$ for $n_{k-1} < n \leq n_k$.
Then  $\xi: = ~<\xi_n>~ \in (\ell^1)^* $, $\xi$ does not satisfy the $\Delta_{1/2}$-condition and $(\xi) \ne (\xi_{a_\infty})$, 
but $\xi_{a_\infty}$ satisfies the $\Delta_{1/2}$-condition if and only if $\epsilon_kn_k = O(\sum_{j=k+1}^\infty \epsilon_j n_j)$.
\item[(iv)] $\xi = o(\xi_{a_\infty})$ and $\xi$ satisfies the $\Delta_{1/2}$-condition for $\xi$ any of the sequences $\omega/log^p$, $\omega/log\,(log\,log)^p$, $\omega/log\,(log\,log)(log\,log\,log)^p,\dots$ 
($p>1$).
\end{example}

\begin{proof}
The verification of (i), (ii) and (iv) is left to the reader. 
For (iii) let us note that if $\frac{\epsilon_{k-1}}{\epsilon_k} \leq M$ for some constant $M$ and all $k > 1$, 
then $\epsilon_k n_k > \frac{k!}{M^{k-1}}\,\epsilon_1 n_1$ which is impossible because $\epsilon_k n_k$ is summable. 
Thus the ratios $\frac{\epsilon_{k-1}}{\epsilon_k}$ are unbounded and hence $\xi$ does not satisfy the  $\Delta_{1/2}$-condition. 
Moreover, $\xi_{a_\infty} \ne O(D_m\xi)$ for every $m$. 
Indeed for every pair of integers $m,p > 1$, choose $k = m^2p^2$.  
Since $n_k > kn_{k-1} > pn_{k-1} > n_{k-1}$, then
\begin{align*}
(\frac{\xi_{a_\infty}}{D_m\xi})_{mpn_{k-1}} &= \frac{(\xi_{a_\infty})_{mpn_{k-1}}}{\xi_{pn_{k-1}}} 
= \frac{1}{\epsilon_k mpn_{k-1}} \sum_{mpn_{k-1}+1}^\infty \xi_j \\
&\geq \frac{1}{\epsilon_k mpn_{k-1}} \sum_{mpn_{k-1}+1}^{kn_{k-1}} \xi_j = \frac{m^2p^2n_{k-1}-mpn_{k-1}}{mpn_{k-1}} = mp -1.
\end{align*}
Thus $\xi_{a_\infty} \neq O(D_m\xi)$ and hence $\xi_{a_\infty} \notin \Sigma((\xi))$. 
Finally, it is straightforward to verify that the given condition, $\epsilon_k =  O(\frac{1}{n_k} \sum_{j=k+1}^\infty \epsilon_j n_j)$, 
is equivalent to the condition $\xi = O(\xi_{a_\infty})$, 
and hence by Corollary \ref{C:(xi) subset (xiainfty)}(ii), 
is equivalent to $\xi_{a_\infty}$ satisfying the  $\Delta_{1/2}$-condition. 
\end{proof}

Examples (i) and (ii) are sequences regular at infinity while (iii) and (iv) are not (see Definition \ref{D:am-infty stability/regularity} and Theorem \ref{T: Theorem 4.12}). 

An immediate consequence of Lemma \ref{L:relations between Dmxi,(Dmxi)ainfty,Dmxiainfty}(i) and (ii) is that the following definition yields characteristic sets.

\begin{definition}\label{D:ainftyI,Iainfty}
Let $I$ be an ideal. Then $_{a_\infty}I$ when $I\neq\{0\}$ and $I_{a_\infty}$ when $I$ is arbitrary are the ideals with characteristic sets
\[
\Sigma(_{a_\infty}I) := \{\xi \in (\ell^1)^* \mid \xi_{a_\infty} \in \Sigma(I)\}
\]
\[
\Sigma(I_{a_\infty}) := \{\xi \in \text{c}_{\text{o}}^{*} \mid \xi = O(\eta_{a_\infty})  ~\text{for some}~ \eta \in \Sigma(I \cap \mathscr L_1)\}.
\]
\end{definition}
\noindent Notice that 
$_{a_\infty}I \subset \mathscr L_1$ and $I_{a_\infty} = (I \cap \mathscr L_1)_{a_\infty}$ by definition and   
that for all $\xi \in (\ell^1)^*$ one has $\xi_{a_\infty} = o(\omega)$, i.e., $I_{a_\infty} \subset se(\omega)$ for every ideal $I$ 
and therefore $_{a_\infty}I =~ _{a_\infty}(I \cap se(\omega))$.
In particular,  $\mathscr L_1 \subset ~_{a_\infty}(se(\omega))$ and hence  $\mathscr L_1 = ~_{a_\infty}(se(\omega))$.
And like the arithmetic and pre-arithmetic mean, the arithmetic mean and pre-arithmetic mean at $\infty$ are inclusion preserving.

\begin{lemma}\label{L:(xi) subset (xiainfty)}
For $\xi \in \text{c}_{\text{o}}^{*}$, 
\[
(\xi)_{a_\infty} = 
\begin{cases}
(\xi_{a_\infty}) ~\text{if}~ \xi \in \ell^1 \\
se(\omega)  ~\text{if}~ \xi \notin \ell^1. 
\end{cases}
\]
In particular, if an ideal $I \not \subset \mathscr L_1$ then $I_{a_\infty}=se(\omega)$, and moreover $(\mathscr L_1)_{a_\infty} = se(\omega)$.
\end{lemma}

\begin{proof}
Assume first that $\xi \in \ell^1$. By definition
$\xi_{a_\infty} \in \Sigma((\xi)_{a_\infty})$ so 
$(\xi_{a_\infty}) \subset (\xi)_{a_\infty}$.  \linebreak
For the reverse inclusion, if $\eta \in \Sigma((\xi)_{a_\infty})$ then 
$\eta \leq \zeta_{a_\infty}$ for a summable $\zeta \in \Sigma((\xi))$.
Since 
$\zeta \leq MD_m\xi$ for some $M>0$ and $m \in \mathbb N$, 
by Lemma \ref{L:relations between Dmxi,(Dmxi)ainfty,Dmxiainfty}~(ii), 
\begin{center}$(\zeta_{a_\infty})_j \leq M((D_m\xi)_{a_\infty})_j \leq M(D_{m+1}\xi_{a_\infty})_j$\end{center} 
for $j \geq m(m-1)$.  
Since $D_{m+1}\xi_{a_\infty} \in \Sigma((\xi_{a_\infty}))$, $\eta \in \Sigma((\xi_{a_\infty}))$ 
and so $(\xi)_{a_\infty} \subset (\xi_{a_\infty})$.

Assume now that $\xi \notin \ell^1$.  
It is not hard to show that $\min(\xi, \omega) \notin \ell^1$,  
so by passing if necessary to a sequence $(\ell^1)^* \not\owns \xi' = o(\min(\xi, \omega))$, 
one can assume without loss of generality that $\xi = o(\omega)$.
For each $\zeta \in \Sigma(se(\omega))$, 
by passing if necessary to $\zeta' := \omega ~uni~(\frac{\zeta}{\omega}) \geq \zeta$ where $uni~\gamma$ is the smallest monotone nonincreasing sequence majorizing $\gamma$ and is given by 
$(uni~\gamma)_n := \underset{j \geq n} {\sup} \gamma_j$, 
one can assume without loss of generality that $\zeta = \alpha \omega$ for some $\alpha \in \text{c}_{\text{o}}^{*}$ and that $\alpha_1 \geq 1$. 
To prove that $\zeta \in \Sigma((\xi)_{a_\infty})$, 
set $m_o = 0$ and choose $n_1 \geq1$ so that $n_1\xi_{n_1} \leq \frac{1}{2}$ and $\alpha_{n_1} \leq \frac{1}{2}$.  
Since $\xi$ is not summable, 
choose the first integer $m_1 \geq n_1$ for which $\sum_{j=n_1}^{m_1} \xi_j \geq \alpha_1$, 
and since  
$\xi_{n_1} \leq \frac{1}{2}$ one also has $\sum_{j=n_1}^{m_1} \xi_j \leq \alpha_1 + \frac{1}{2}$. 
Now choose $n_2>m_1$ and $m_2 \geq n_2$ so that $n_2\xi_{n_2} \leq \frac{1}{2^2},~\alpha_{n_2} \leq \frac{1}{2^2}$ 
and $\frac{1}{2} \leq \sum_{j=n_2}^{m_2} \xi_j \leq 1$. 
Iterating obtains the sequences, 
$m_k \geq n_k > m_{k-1}$ so that, 
for $k\geq2$, $n_k\xi_{n_k} \leq \frac{1}{2^k},~\alpha_{n_k} \leq \frac{1}{2^k}$  
and $\frac{1}{2^{k-1}} \leq \sum_{j=n_k}^{m_k} \xi_j \leq \frac{1}{2^{k-2}}$. \linebreak
Define 
$\eta_n = 
\begin{cases} \xi_{n_k}	&\text{for $m_{k-1} < n \leq n_k$} \\
\xi_n	&\text{for $n_k \leq n \leq m_k$}
\end{cases}
$. 
Then $\eta := ~<\eta_n>~ \in \text{c}_{\text{o}}^{*}$ and $\eta \leq \xi$.  
Since 
$
\sum_{j=m_{k-1}+1}^{n_k-1} \eta_j = (n_k-m_{k-1}-1)\xi_{n_k} < n_k\xi_{n_k} \leq \frac{1}{2^k}
$ 
and $\sum_{j=n_k}^{m_k} \eta_j = \sum_{j=n_k}^{m_k} \xi_j \leq \frac{1}{2^{k-2}}$, 
one has $\eta \in \ell^1$.  
Moreover, for $k \geq 2$ and 
$n_{k-1} \leq n < n_k$ one has 
\begin{center}$\sum_{n+1}^{\infty} \eta_j \geq \sum_{n_k}^{m_k} \xi_j \geq \frac{1}{2^{k-1}} \geq \alpha_{n_{k-1}} \geq \alpha_n$; \end{center}
for $1 \leq n < n_1$ one has $\sum_{n+1}^{\infty} \eta_j \geq \sum_{n_1}^{m_1} \xi_j \geq \alpha_1 \geq \alpha_n$. 
Thus $\eta \in (\ell^1)^*$, $\eta \leq \xi$ and 
$\alpha \leq~ <\sum_{n+1}^{\infty} \eta_j>$. 
Therefore $\zeta = \alpha \omega \leq \eta_{a_\infty}$ and hence $\zeta \in \Sigma((\xi)_{a_\infty})$.
Since $\zeta \in \Sigma(se(\omega))$ was arbitrary, $(\xi)_{a_\infty} \supset se(\omega)$. 
But $I_{a_\infty} \subset se(\omega)$ for every ideal $I$, so one has equality. 
Thus $I_{a_\infty}= se(\omega)$ when $I \not\subset \mathscr L_1$.
Moreover, \begin{center}$(\mathscr L_1)_{a_\infty} = (\mathscr L_1 \cap (\omega))_{a_\infty} = (\omega)_{a_\infty}= se(\omega)$. \end{center}
\end{proof}

Notice that to prove directly that $se(\omega) \subset (\mathscr L_1)_{a_\infty}$, 
it would only be necessary to show that for all $\alpha \in \text{c}_{\text{o}}^{*}$, 
$\alpha \leq~ <\sum_{n+1}^{\infty} \eta_j>$ for some $\eta \in (\ell^1)^*$.  
This is 
equivalent to the well-known fact that $\alpha$ has a convex majorant in $\text{c}_{\text{o}}^{*}$ 
(see \cite[p. 203]{nB64}).  
But for the proof that $se(\omega) \subset (\xi)_{a_\infty}$ we needed to prove that the convex majorant, $<\sum_{n+1}^{\infty} \eta_j>$, of $\alpha$ can be chosen so that additionally $\eta \leq \xi$.

Recall from Section 2 (see the paragraph following Theorem \ref{T:DFWW}) that the am-ideals satisfy the 5-chain of inclusions. 
The situation is slightly more complicated for the am-$\infty$ case since 
the inclusion $I \subset I_{a_\infty}$ holds if and only if $I \subset se(\omega)$, as we shall see in the next proposition.
We shall also see there that the 5-chain of inclusions
remains valid for all ideals $I$ contained in $\mathscr L_1$:
\[
_{a_\infty}I \subset (_{a_\infty}I)_{a_\infty} \subset I \subset ~_{a_\infty}(I_{a_\infty}) \subset I_{a_\infty}.
\]
\noindent More generally, 

\begin{proposition}\label{P:aminfty relations}
Let $I\neq\{0\}$ be an ideal. 
\item[(i)] $\{0\} \neq~ _{a_\infty}I \subset (_{a_\infty}I)_{a_\infty} \subset I$
\item[(i$'$)] $I \cap \mathscr L_1 \subset ~_{a_\infty}(I_{a_\infty}) \subset I_{a_\infty}$
\item[(ii)] $_{a_\infty}I =~ _{a_\infty}((_{a_\infty}I)_{a_\infty})$ and the map 
$I \rightarrow (_{a_\infty}I)_{a_\infty}$ is idempotent. 
\item[(ii$'$)] $I_{a_\infty} = (_{a_\infty}(I_{a_\infty}))_{a_\infty}$ and the map 
$I \rightarrow~ _{a_\infty}(I_{a_\infty})$ is idempotent. 

\item[(iii)] If $J$ is an ideal, then $J_{a_\infty} \subset I$ if and only if $J \cap \mathscr L_1 \subset~ _{a_\infty}I$.

\item[(iv)] $I \subset I_{a_\infty}$ if and only if $I \subset se(\omega)$.
\end{proposition}

\begin{proof}  
\item[(i) and (i$'$)]
The inclusions $_{a_\infty}I \subset I$ and 
$I \cap \mathscr L_1 \subset I_{a_\infty}$ follow from Definition \ref{D:ainftyI,Iainfty} and Corollary \ref{C:(xi)subset(xi ainfty)}. 
Applying the first inclusion to $I_{a_\infty}$ and the second to $_{a_\infty}I$ obtains $_{a_\infty}(I_{a_\infty}) \subset I_{a_\infty}$ 
and $_{a_\infty}I =~ _{a_\infty}I \cap \mathscr L_1 \subset (_{a_\infty}I)_{a_\infty}$.
The remaining two inclusions follow directly from Definition \ref{D:ainftyI,Iainfty}. 
And since $<1,0,\dots>_{a_\infty} = 0$, $_{a_\infty}I \neq \{0\}$ for all $I \ne \{0\}$.
\item[(ii)] By (i) and (i$'$), 
$_{a_\infty}((_{a_\infty}I)_{a_\infty}) \subset~ _{a_\infty}I =~ _{a_\infty}I \cap \mathscr L_1  \subset~ _{a_\infty}((_{a_\infty}I)_{a_\infty})$. 
\item[(ii$'$)] Again by (i$'$) and (i), 
$I_{a_\infty} = (I \cap \mathscr L_1)_{a_\infty} \subset~ (_{a_\infty}(I_{a_\infty}))_{a_\infty} \subset I_{a_\infty}$. 
\item[(iii)] 
If $J_{a_\infty} \subset I$ then, by (i$'$), one has $J \cap \mathscr L_1 \subset~_{a_\infty}(J_{a_\infty}) \subset~_{a_\infty}I$.
Conversely, if $J \cap \mathscr L_1 \subset~ _{a_\infty}I$ then, 
by (i) and the paragraph preceding Lemma \ref{L:(xi) subset (xiainfty)}, one has
$J_{a_\infty} =(J \cap \mathscr L_1)_{a_\infty} \subset (_{a_\infty}I)_{a_\infty} \subset I$.

\item[(iv)] 
That $I \cap se(\omega) \subset I_{a_\infty}$ is a simple consequence of Corollary \ref{C:(xi)subset(xi ainfty)} and Lemma \ref{L:(xi) subset (xiainfty)}. 
In particular, $I \subset se(\omega) ~\text{implies}~ I \subset I_{a_\infty}$. 
The converse implication is automatic since, by definition, $I_{a_\infty} \subset se(\omega)$.
\end{proof}

Immediate consequences of Proposition \ref{P:aminfty relations}(iii), Lemma \ref{L:(xi) subset (xiainfty)}, 
and the identities $sc~se(\omega) = sc(\omega) = (\omega)$ and $se~scI = seI \subset I$ are:

\begin{corollary}\label{C: ainftyI = L1 if and only if se(omega) subset I} Let $I\ne \{0\}$ be an ideal. 
\item[(i)] $_{a_\infty}I = \mathscr L_1$ if and only if $se(\omega) \subset I$ if and only if $\omega \in \Sigma(scI)$.
\item[(ii)] $I_{a_\infty} = se(\omega)$ if and only if $\mathscr L_1 =\,_{a_\infty}(I_{a_\infty})$. 
\end{corollary}

\textit{Am-$\infty$ stability}, the analog of am-stability, is defined for nonzero ideals by any of the following equivalent conditions.
  
\begin{corollary}\label{C: a-infty stability equivalences}
Let $I \neq \{0\}$ be an ideal. The following are equivalent.
\item[(i)] $I=~_{a_\infty}I$
\item[(ii)] $I \cap \mathscr L_1 = I_{a_\infty}$
\item[(iii)] $I \subset \mathscr L_1$ and $I=I_{a_\infty}$
\end{corollary}

\begin{proof}  
\item[(i)$\Rightarrow$ (ii)] 
Since $I =~_{a_\infty}I \subset \mathscr L_1$, 
$I = I \cap \mathscr L_1 \subset I_{a_\infty}$ by Proposition \ref{P:aminfty relations}(i$'$) and the reverse inclusion follows by Proposition \ref{P:aminfty relations}(i).
\item[(ii)$\Rightarrow$ (iii)]
If $I \not\subset \mathscr L_1$, by Lemma \ref{L:(xi) subset (xiainfty)}, $I_{a_\infty} = se(\omega) \not\subset \mathscr L_1$, against (ii). 
\item[(iii)$\Rightarrow$ (i)] 
One has $I = I \cap \mathscr L_1 \subset~ _{a_\infty}(I_{a_\infty}) =~_{a_\infty}I \subset I$ by Proposition \ref{P:aminfty relations}(i$'$) and (i).
\end{proof}
\begin{definition}\label{D:am-infty stability/regularity}
An ideal $I \neq \{0\}$ is called am-stable at infinity 
(or am-$\infty$ stable) if $I=~_{a_\infty}I$.  
A sequence $\xi \in (\ell^1)^*$ is called regular 
at infinity ($\infty$-regular for short) if $(\xi)=~_{a_\infty}(\xi)$.
\end{definition}

\noindent Therefore, $(\xi)$ is $\infty$-regular if and only if 
$(\xi)=(\xi_{a_\infty})$ 
by Corollary \ref{C: a-infty stability equivalences} and Lemma \ref{L:(xi) subset (xiainfty)}, 
if and only if $\xi_{a_\infty} = O(D_m\xi)$ for some $m \in \mathbb N$ by Corollary \ref{C:(xi)subset(xi ainfty)} 
(cf. \cite[Corollary 5.6(c)]{mW02}), 
and surprisingly and more simply, if and only if $\xi_{a_\infty} = O(\xi)$ (see Theorem \ref{T: Theorem 4.12} below). 
The notion of regularity at infinity for summable sequences is 
an analog of the usual notion of regularity of nonsummable sequences that was used extensively in \cite{GK69} and that plays a key role also in Varga's construction of positive traces on principal ideals \cite{jV89}.

Several characterizations of regular sequences in the am-case have analogs in the am-$\infty$ case (Theorem \ref{T: Theorem 4.12} below),
although the proofs have to contend with the problem that $\xi_{a_\infty}$ may not satisfy the $\Delta_{1/2}$ condition or, 
equivalently (Corollary \ref{C:(xi) subset (xiainfty)}), that $\xi$ may not be $O(\xi_{a_\infty})$. 

For convenience we recall the definition of the Matuszewska indices $\alpha(\xi)$ and $\beta(\xi)$ for a monotonic sequence 
$\xi$ (\cite[Section 2.4]{DFWW}):
\[    \alpha(\xi): = \underset{n}{\lim}\,\frac{log\,\overline{\xi}_n}{log\,n} = \underset{n\geq 2}{\inf}\frac{log\,\overline{\xi}_n}{log\,n} 
\quad \text{and}\quad     \beta(\xi):= \underset{n}{\lim}\,\frac{log\,\underline{\xi}_n}{log\,n} = \underset{n\geq 2}{\sup}\frac{log\,\underline{\xi}_n}{log\,n} \]
where $\overline{\xi}_n := \underset{k}{\sup}\frac{\xi_{kn}}{\xi_k}$ and $\underline{\xi}_n := \underset{k}{\inf}\frac{\xi_{kn}}{\xi_k}$. 
It can be shown that
\[
\alpha(\xi) = \inf \{\gamma \mid \exists\,C>0~\text{such that}~ \xi_n \leq C\left(\frac{n}{m}\right)^\gamma \xi_m ~\text{for all}~ n \geq m\} 
\]
\[
\beta(\xi) = \sup \{\gamma \mid \exists \,C>0~\text{such that}~ \xi_n \geq C\left(\frac{n}{m}\right)^\gamma \xi_m ~\text{for all}~ n \geq m \}.
\]
The above inequalities characterizing the Matuszeswka indices are the discrete analog of the Potter type inequalities in the theory of functions of regular and O-regular variation (cf. \cite[Proposition 2.2.1]{BGT89}), 
and were linked to regularity of $\text{c}_\text{o}^*$-sequences in \cite[Theorem 3.10]{DFWW},
where it was proven that a sequence $\xi \in \text{c}_{\text{o}}^*$  is regular if and only if $\beta(\xi) > -1$ if and only if $\xi_a$ is regular. 
As indicated in \cite[Remark 3.11]{DFWW}, 
the equivalence of $\xi$ regularity and $\xi_a$ regularity is also implicit in the work of Varga \cite[Theorem IRR]{jV89}.

\begin{theorem}\label{T: Theorem 4.12}
If $\xi \in (\ell^1)^*$, the following conditions are equivalent.
\item[(i)]   $\xi$ is $\infty$-regular.
\item[(ii)]  $\xi_{a_\infty} = O(\xi)$
\item[(iii)] $\alpha(\xi) < -1$, i.e., there are constants $C > 0$ and $p > 1$ for which $\xi_n \leq C (\frac{m}{n})^p \xi_m$ for all integers $n \geq m$.
\item[(iv)]  $\xi_{a_\infty}$ is $\infty$-regular.
\item[(v)]  $\underset{n}{\inf} \frac{(\xi_{a_\infty})_n}{(\xi_{a_\infty})_{kn}} > k$ for some integer $k > 1$
\item[(v$'$)]  $\underset{n}{\inf} \frac{(\xi_{a_\infty})_n}{(\xi_{a_\infty})_{kn}} > k$ for all integers $k > 1$
\item[(v$''$)] $\underset{n}{\inf} \frac{\xi_n}{(\xi_{a_\infty})_{kn}} > 0$ for all integers $k > 1$
\item[(v$'''$)] $\underset{n}{\inf} \frac{\xi_n}{(\xi_{a_\infty})_{kn}} > 0$ for some integer $k > 1$
\item[(vi)] $\underset{k}{\sup}~ \underset{n}{\inf} \frac{(\xi_{a_\infty})_n}{k((\xi_{a_\infty})_{kn}} = \infty$
\end{theorem}

\begin{proof} 
\item[(i) $\Rightarrow$ (ii)] Assume that $\xi$ is regular at infinity, that is, $(\xi) =\, _{a_\infty}(\xi)$. 
Then $(\xi) = (\xi)_{a_\infty} = (\xi_{a_\infty})$ by Corollary \ref{C: a-infty stability equivalences} and Lemma \ref{L:(xi) subset (xiainfty)} 
and therefore $\xi_{a_\infty} \leq MD_m \xi$ for some $m \in \mathbb N$ and $M > 0$. 
In particular, $(\xi_{a_\infty})_{mn} \leq M \xi_n$ for all $n$.
The case $m=1$ is (ii). If $m>1$ then  
\begin{align*}
(\xi_{a_\infty})_{mn} &= \frac{1}{mn} \sum_{mn+1}^{\infty} \xi_i 
= \frac{1}{mn} \{\sum_{mn+1}^{m^2n} \xi_i + \sum_{m^2n+1}^{m^3n}\xi_i + \sum_{m^3n+1}^{m^4n}\xi_i + \cdots \} \\
&\geq \frac{1}{mn} \{(m-1)mn \,\xi_{m^2n} + (m-1)m^2n \, \xi_{m^3n} + (m-1)m^3n \, \xi_{m^4n} + \cdots \} \\
&= \frac{m-1}{m^2} \sum_{k=2}^{\infty} m^k \, \xi_{m^kn}.
\end{align*}
Thus $\frac{m-1}{m^2} \sum_{k=2}^{\infty} m^k \xi_{m^kn} \leq M\xi_n$ for all $n$ and hence, by substituting here $mn$ for $n$,
\[
\frac{m-1}{m^3} \sum_{k=3}^{\infty} m^k \xi_{m^kn} \leq M\xi_{mn}.
\]
On the other hand, 
the same formula for $(\xi_{a_\infty})_{mn}$ yields 
\begin{align*}
(\xi_{a_\infty})_{mn} &\leq \frac{1}{mn} \{(m-1)mn \,\xi_{mn} + (m-1)m^2n \,\xi_{m^2n} + (m-1)m^3n \,\xi_{m^3n} + \cdots\} \\
&= \frac{m-1}{m} (m \,\xi_{mn} + m^2 \,\xi_{m^2n}) + \frac{m-1}{m} \sum_{k=3}^{\infty} m^k \,\xi_{m^kn} \\
&\leq \frac{m-1}{m} (m + m^2) \,\xi_{mn} + m^2 M \xi_{mn} = M'\xi_{mn}.
\end{align*}
From this, since for each $j \in \mathbb N$, $j=mn-p$ for some $n \in \mathbb N$ and some $0 \leq p \leq m-1$, one obtains
\begin{align*} 
(\xi_{a_\infty})_j &= \frac{1}{mn-p} \{ \xi_{mn-p+1} + \xi_{mn-p+2} + \cdots + \xi_{mn} + \sum_{i=mn+1}^{\infty} \xi_i \} \\
&\leq \frac{p}{mn-p} \xi_{mn-p} + \frac{mn}{mn-p}(\xi_{a_\infty})_{mn} \\
&\leq \frac{p}{mn-p} \xi_{mn-p} +  \frac{mn}{mn-p} M'\xi_{mn} \\
&\leq (\frac{p}{mn-p} +  \frac{mn}{mn-p}M')\,\xi_{mn-p} \\
&\leq (m-1 + mM')\,\xi_j,
\end{align*}
which concludes the proof.
\item[(ii) $\Rightarrow$ (i)]  Obvious from remarks following Definition \ref{D:am-infty stability/regularity}.
\item[(ii) $\Rightarrow$ (iii)] Let $\xi_{a_\infty} \leq M\xi$ for some $M > 0$ and 
without loss of generality assume that $\xi_n > 0$ for all $n$. 
From the basic identity 
\[(n-1)(\xi_{a_\infty})_{n-1} =  \xi_n + n(\xi_{a_\infty})_n\] follows the recurrence
\[
(\xi_{a_\infty})_n = \frac{\frac{n-1}{n}}{1+ \frac{1}{n} (\frac{\xi}{\xi_{a_\infty}})_n} (\xi_{a_\infty})_{n-1}
\]
and hence for all $n > m \geq 1$,
\begin{align*}
(\xi_{a_\infty})_n &= \frac{\frac{m}{n}}{\prod_{j=m+1}^{n} (1+\frac{1}{j} (\frac{\xi}{\xi_{a_\infty}})_j)} (\xi_{a_\infty})_m \\
        &\leq \frac{\frac{m}{n}}{\prod_{j=m+1}^{n} (1+\frac{1}{Mj})} (\xi_{a_\infty})_m \\
        &= \frac{m}{n} e^{-\sum_{j=m+1}^{n} log \,(1+\frac{1}{Mj})} (\xi_{a_\infty})_m.
\end{align*}
Let $N$ be the smallest integer larger than or equal to $\frac{1}{M}$, 
$p:=1+\frac{1}{M}$ and set \linebreak
$K:=N^pe^{(\frac{1og\,2}{M}+\frac{1}{2M^2})}$.
If $m \geq N$ then $Mj>1$ for all $j \geq m+1$ and hence $log\,(1+\frac{1}{Mj})> \frac{1}{Mj}-\frac{1}{2M^2j^2}$.
Thus if $n > m \geq N$,
\begin{align*}
(\xi_{a_\infty})_n &\leq \frac{m}{n} e^{-\sum_{j=m+1}^{n} (\frac{1}{Mj}-\frac{1}{2M^2j^2})} (\xi_{a_\infty})_m \\
        &\leq \frac{m}{n} e^{(-\frac{1}{M} \,log\,\frac{n+1}{m+1}\,+\,\frac{1}{2M^2})} (\xi_{a_\infty})_m \\
       &\leq e^{(\frac{log~2}{M}+\frac{1}{2M^2})} (\frac{m}{n})^p (\xi_{a_\infty})_m.
\end{align*}
If $n > N > m$, the above inequality implies
\[
(\xi_{a_\infty})_n \leq   e^{(\frac{log~2}{M}+\frac{1}{2M^2})} (\frac{N}{n})^p   (\xi_{a_\infty})_N  \leq K(\frac{m}{n})^p (\xi_{a_\infty})_m.
\]
If $N \geq n \geq m$ then $K \geq N^p \geq (\frac{n}{m})^p$ and hence 
\[
(\xi_{a_\infty})_n \leq (\xi_{a_\infty})_m \leq  K(\frac{m}{n})^p (\xi_{a_\infty})_m.
\]
Thus $(\xi_{a_\infty})_n \leq  K(\frac{m}{n})^p (\xi_{a_\infty})_m$ for all $n \geq m$, hence
$(\xi_{a_\infty})_n \leq MK(\frac{m}{n})^p \xi_m$.

From Lemma \ref{L:relations between Dmxi,(Dmxi)ainfty,Dmxiainfty}(iv), 
$\xi_{2n} \leq (\xi_{a_\infty})_n$ so
\[
\xi_{2n} \leq MK(\frac{m}{n})^p \xi_m
\]
for all $n \geq m$.
Set $C = 3^pMK$
and let $k \geq 2m$. If $k$ is even, then
\[
\xi_k \leq MK(\frac{m}{k/2})^p \xi_m \leq C (\frac{m}{k})^p \xi_m,
\]
while if $k$ is odd then
\[
\xi_k \leq \xi_{n-1} \leq MK(\frac{m}{(k-1)/2})^p \xi_m \leq  3^p MK(\frac{m}{k})^p \xi_m = C (\frac{m}{k})^p \xi_m.
\]
Finally, if $m \leq k < 2m$, $\xi_k \leq \xi_m < 2^p(\frac{m}{k})^p\xi_m \leq C(\frac{m}{k})^p\xi_m$
since $MK > M2^{1/M} > 1$.

\item[(iii) $\Rightarrow$ (ii)] 
A direct computation shows that 
$(\xi_{a_\infty})_n \leq \frac{C}{p-1}\xi_n$.
\item[(ii) $\Rightarrow$ (iv)] If $\xi_{a_\infty} \leq M\xi$, then  $\xi_{a_\infty^2} \leq M\xi_{a_\infty}$, 
hence by the equivalence of (i) and (ii), $\xi_{a_\infty}$ is $\infty$-regular.
\item[(iv) $\Rightarrow$ (ii)] Since (i) and (iii) are equivalent, there exists $p > 1,\,C>0$ so \\
$\left(\xi_{a_\infty}\right)_n \leq C\left(\frac{m}{n}\right)^p\left(\xi_{a_\infty}\right)_m$ for all $n \geq m$. 
Thus 
$\left(\xi_{a_\infty}\right)_{km} \leq \left(\frac{1}{k}\right)^q\left(\xi_{a_\infty}\right)_m$ for some \linebreak
$q > 1$ and integer $k > 1$, 
and hence $\sum_{km+1}^\infty \xi_i \leq \left(\frac{1}{k}\right)^{q-1}\sum_{m+1}^\infty \xi_i$ for all $m$. 
But then $\left(1-(\frac{1}{k})^{q-1}\right)  \sum_{m+1}^\infty \xi_i\leq \sum_{m+1}^{km} \xi_i \leq (k - 1)m\xi_m$ and hence (ii) holds.
\item[(ii) $\Rightarrow$ (v$'$) $\Rightarrow$ (v$''$) $\Rightarrow$ (v$'''$)] For every $k > 1$, 
\[(k-1)n \xi_{kn} \leq \sum_{j=n+1}^{kn} \xi_j \leq (k-1)n\xi_n \quad \text{for all $n$,} \] 
hence
\[
1+\frac{(k-1)\xi_{kn}}{k(\xi_{a_\infty})_{kn}} \leq  \frac{\sum_{n+1}^{\infty} \xi_j}{\sum_{kn+1}^\infty \xi_j} \leq 1+\frac{(k-1)\xi_n}{k(\xi_{a_\infty})_{kn}} 
\]
or, equivalently,
\[
(k-1)\frac{\xi_{kn}}{(\xi_{a_\infty})_{kn}}  \leq   \frac{(\xi_{a_\infty})_{n}}{(\xi_{a_\infty})_{kn}}-k   \leq (k-1)\frac{\xi_n}{(\xi_{a_\infty})_{kn}}.
\]
Thus, for every $k > 1$,
\[
\xi_{a_\infty} = O(\xi) \Rightarrow \underset{n}{\inf}\frac{\xi_{kn}}{(\xi_{a_\infty})_{kn}} > 0 \Rightarrow \underset{n}{\inf}\frac{(\xi_{a_\infty})_n}{(\xi_{a_\infty})_{kn}} > k \Rightarrow \underset{n}{\inf} \frac{\xi_n}{(\xi_{a_\infty})_{kn}} > 0.
\]
\item[(v$'$) $\Rightarrow$ (v) $\Rightarrow$ (v$'''$) $\Rightarrow$ (i)]
The first implication is obvious and the second follows from the same double inequality we used above.
If (iv$'''$) holds, i.e.,  for some $M>0$ and $k > 1$,  $(\xi_{a_\infty})_{kn} \leq M\xi_n$ for all $n$, 
then for $j \in \mathbb N$, $j=kn-p$ with $0 \leq p \leq k-1$ and
\begin{align*}
(\xi_{a_\infty})_j &= \frac{1}{kn-p} (\xi_{kn-p+1} + \cdots + \xi_{kn} + kn(\xi_{a_\infty})_{kn}) \\
&\leq  \frac{1}{kn-p} (p\xi_n + knM\xi_n)  \\
&\leq  (k - 1 + kM)\xi_n  \\
&=  (k - 1 + kM) (D_k \xi)_j.
\end{align*}
Thus $\xi_{a_\infty} = O(D_k\xi)$, that is, $\xi$ is $\infty$-regular.
\item[(v$'$) $\Rightarrow$ (vi)] $\underset{k}{\sup}~ \underset{n}{\inf} \frac{(\xi_{a_\infty})_n}{k(\xi_{a_\infty})_{kn}} \geq 1$ 
since $n(\xi_{a_\infty})_n \geq kn(\xi_{a_\infty})_{kn}$ for all $n,k$.
Suppose  
$\underset{k}{\sup}~ \underset{n}{\inf} \frac{(\xi_{a_\infty})_n}{k(\xi_{a_\infty})_{kn}} = M < \infty$. 
Then $M = 1$ since otherwise, for some $k$, 
$\underset{n}{\inf} \frac{(\xi_{a_\infty})_n}{k(\xi_{a_\infty})_{kn}} > \sqrt {M}$ and hence 
$
\underset{n}{\inf} \frac{(\xi_{a_\infty})_n}{k^2(\xi_{a_\infty})_{k^2n}} 
\geq \underset{n}{\inf} \frac{(\xi_{a_\infty})_n}{k(\xi_{a_\infty})_{kn}}~ \underset{n}{\inf} \frac{(\xi_{a_\infty})_{kn}}{k(\xi_{a_\infty})_{k^2n}} > M
$
against the definition of $M$. But $M = 1$ contradicts (v$'$). 
\item[(vi) $\Rightarrow$ (v)] Obvious.
\end{proof}

\begin{remark}\label{R: Aljancic/Kalton}
\item[(a)] The Potter type inequality in (iii) 
(cf. \cite[Proposition 2.2.1]{BGT89}, and see also \cite{Aljan-Aran77} and \cite[Theorem 3.10, Remark 3.11]{DFWW})
was shown by Kalton in \cite[Corollary 7]{nK87} to be necessary and sufficient for $(\xi)$ 
to support a unique separately continuous trace. 
By Theorem \ref{T: Theorem 4.12} and  Theorem \ref{T:i-iii} below this condition is also necessary and sufficient for $(\xi)$ to support a unique trace, 
which in this case coincides with $Tr$ and hence is separately continuous.
\item[(b)] In the course of the proof of (ii) $\Rightarrow$ (iii) we have obtained that if $\xi$ is $\infty$-regular and if $\xi_{a_\infty} \leq M\xi$, 
then $\alpha(\xi) \leq -\left(1+\frac{1}{M}\right)$. 
Hence $\alpha(\xi) \leq -1 - inf\,\frac{\xi}{\xi_{a_\infty}}$.
\item[(c)] The proof of  (iv) $\Rightarrow$ (ii) is an adaptation of the proof in the am-case in \cite[Theorem 3.10: (c)$'$ $\Rightarrow$ (a)]{DFWW}. 
Conditions (v)-(v$'''$) and their proofs are am-$\infty$ analogues of the characterizations of regular sequences given by Varga in \cite[Lemma 1]{jV89}. Albeverio et al. in \cite{AGPS87} used conditions equivalent to the negation of (v) and (vi) to define ``generalized eccentric" operators (for the trace class case) for which their main result showed the existence of positive singular traces.
\item[(d)] 
Condition (ii) strengthens \cite[Corollary 5.19]{DFWW} and \cite[Corollary 5.6]{mW02} by eliminating the need for ampliations, i.e., replacing the condition $\xi_{a_\infty} = O(D_m\xi)$ for some $m$ by the condition $\xi_{a_\infty} = O(\xi)$.
\item[(d)] Whereas regular sequences are those for which $\xi \asymp \xi_a$, this is not true for the am-$\infty$ case.
In fact, by Corollary \ref{C:(xi) subset (xiainfty)}(ii), $\xi \asymp \xi_{a_\infty}$ if and only if $\xi$ is $\infty$-regular 
(i.e., $(\xi)=(\xi_{a_\infty})$) and satisfies the $\Delta_{1/2}$-condition. 
And as Example \ref{E:infty regulars/irregulars}(ii) shows, $\xi$ can be regular at infinity while $\xi \not\asymp \xi_{a_\infty}$. 
\end{remark}

The equivalence of (ii) and (iii) is the am-$\infty$ analog of \cite[Theorem 3.10(a),(b),(b)$'$]{DFWW}.
We give a direct proof of the am-case that provides also a lower bound for $\beta(\xi)$.

\begin{proposition}\label{P: Aljancic/Kalton} A sequence $\xi \in \text{c}_{o}^*$ is regular, 
i.e., $\xi_a = O(\xi)$, 
 if and only if there are constants $C > 0$ and $0 < p < 1$ for which $\xi_n \geq C(\frac{m}{n})^p \,\xi_m$ for all $n \geq m$, and then $\beta(\xi)  \geq -1 + inf\,\frac{\xi}{\xi_a}$.
\end{proposition}

\begin{proof} A simple computation shows that if $\xi_n \geq C(\frac{m}{n})^p\,\xi_m$ for all $n \geq m$, 
\linebreak then $\xi_a \leq \frac{1}{(1-p)C}\,\xi$. 

Conversely, assume $\xi\ne 0$ and $\xi_a = O(\xi)$, i.e., $\xi_a \leq M\xi$ for some $M>1$.
The identity $n(\xi_a)_n = \xi_n + (n-1)(\xi_a)_{n-1}$ implies the recurrence
\[
(\xi_a)_n = \frac{\frac{n-1}{n}}{1-\frac{1}{n}(\frac{\xi}{\xi_a})_n} (\xi_a)_{n-1}
\]
and hence 
\[
(\xi_a)_n = \frac{\frac{m}{n}}{\prod_{j=m+1}^{n}(1-\frac{1}{j}(\frac{\xi}{\xi_a})_j)}(\xi_a)_m
\] 
 for all $n > m$. Then
\begin{align*}
\xi_n &= \frac{\frac{m}{n}(\frac{\xi}{\xi_a})_n}{\prod_{j=m+1}^{n}(1-\frac{1}{j}(\frac{\xi}{\xi_a})_j)}~(\xi_a)_m 
\geq \frac{\frac{m}{Mn}}{\prod_{j=m+1}^{n}(1-\frac{1}{Mj})}~\xi_m\\
&= \frac{m}{Mn}~e^{-\sum_{j=m+1}^{n}log\,(1-\frac{1}{Mj})}~\xi_m 
\geq \frac{m}{Mn} e^{\sum_{j=m+1}^{n}\frac{1}{Mj}}~\xi_m\\
&\geq \frac{m}{Mn} e^{\frac{1}{M}(log\,\frac{n}{m}-log\,2)}~\xi_m = \frac{1}{M2^\frac{1}{M}}\,(\frac{m}{n})^{1-\frac{1}{M}}~\xi_m.
\end{align*}
And from the inequality characterizing the Matuszewska index $\beta(\xi)$ mentioned prior to Theorem \ref{T: Theorem 4.12}, $\beta(\xi) \geq -1+\frac{1}{M}$,
from which it follows that $\beta(\xi)  \geq -1 + inf\,\frac{\xi}{\xi_a}$.
\end{proof}

For the readers' convenience, we summarize the known relations between the sequence properties most frequently considered
in this paper, the Matuszewska index $\beta$, and the new relations to the analogous properties for Matuszewska's index $\alpha$ developed here. 
\begin{corollary}\label{C: Matusz} Let $\xi\in\text{c}_\text{o}^*$. Then $-\infty \leq \beta(\xi) \leq \alpha(\xi) \leq 0$ and
\item[(i)] $\xi$ satisfies the $\Delta_{1/2}$-condition if and only if $\beta(\xi) > -\infty$, i.e.,
 if and only if there are constants $C>0$ and $p>0$ for which $\xi_n \geq C(\frac{m}{n})^p \xi_m$ for all $n\geq m$,
 if and only if $\xi^e$ is regular for some $e>0$. (\cite[2.4~ (22), 2.23, Theorem 3.5]{DFWW})
\item[(ii)] If $\xi \asymp \eta_a$ for some $\eta\in \text{c}_\text{o}^*$, then $\beta(\xi) \geq -1$ (since for $n\geq m$, $\frac{(\xi_a)_n}{(\xi_a)_m}\geq \frac{m}{n}$).
\item[(iii)] $\xi$ is regular if and only if $\beta(\xi) > -1$. (\cite[Theorem 3.10]{DFWW})
\item[(iv)] $\xi^s$ is regular for every $s>0$ if and only if $\beta(\xi) =0$. (\cite[Corollary 5.16]{DFWW})
\item[(i$'$)] $\xi$ satisfies the condition $\sup \frac{\xi_{2n}}{\xi_n}<1$ if and only if $\alpha(\xi) < 0$, i.e.,
 if and only if there are constants $C>0$ and $p>0$ for which $\xi_n \leq C(\frac{m}{n})^p \xi_m$ for all $n\geq m$,
 if and only if $\xi^e$ is $\infty$-regular for some $e>0$. (Elementary from the definition and (iii$'$.)
\item[(ii$'$)] If $\xi \asymp \eta_{a_\infty}$ for some $\eta\in \text{c}_\text{o}^*$, 
then $\alpha(\xi) \leq -1$ (since for $n\geq m$, $\frac{(\xi_{a_\infty})_n}{(\xi_{a_\infty})_m}\leq \frac{m}{n}$).
\item[(iii$'$)] $\xi$ is $\infty$-regular if and only if $\alpha(\xi)< -1$. (Theorem \ref{T: Theorem 4.12})
\item[(iv$'$)] $\xi^s$ is $\infty$-regular for every $s>0$ if and only if $\alpha(\xi) = -\infty$. (By (iii$'$).)
\end{corollary}

In another paper we study lattice properties of am-$\infty$ stable ideals. Among other results there we show
that every principal ideal with the exception of the finite rank ideal $F$ contains a am-$\infty$ stable principal ideal strictly larger than $F$ and is contained in an am-stable principal ideal. 
$F$ is the smallest nonzero am-$\infty$ stable ideal, $K(H)$ is the largest am-stable ideal and  
there is a largest am-$\infty$ stable ideal $st_{a_\infty}(\mathscr L_1)$ and a smallest am-stable ideal 
$st^a(\mathscr L_1)$ (see below).
These naturally divide all ideals into the three classes described in the Introduction, namely, 
the ``small ideals" contained in $st_{a_\infty}(\mathscr L_1)$,
the ``large ideals" containing $st^a(\mathscr L_1)$, and the ``intermediate ideals" that are neither.

\begin{definition}\label{D: stabilizers}
The lower and upper am-stabilizers (resp., am-$\infty$ stabilizers) for an ideal $I$ are:
\begin{align*} 
st_a(I) &:= \bigcap_{m=0}^{\infty} ~_{a^m}I \\
st^a(I) &:= \bigcup_{m=0}^\infty I_{a^m} \\
st_{a_\infty}(I) &:= \bigcap_{m=0}^\infty ~_{a^m_\infty}I ~\text{for}~ I \ne \{0\} \\
st^{a_\infty}(I) &:= \bigcup_{m=0}^\infty I_{a^m_\infty} ~\text{for}~ I \subset st_{a_\infty}(\mathscr L_1)
\end{align*}
\end{definition}

It is easy to verify that $st_a(I)$ (resp., $st^a(I)$) is the largest am-stable ideal contained in $I$ (resp., the smallest am-stable ideal containing $I$).

It follows similarly from Proposition \ref{P:aminfty relations}(i) that $st_{a_\infty}(I)$ is well-defined and is the largest am-$\infty$ stable ideal contained in $I$. 

If $I \subset se(\omega)$, then $\{I_{a^m_\infty}\}$ is 
an increasing nest of ideals by Proposition \ref{P:aminfty relations}(iv) and hence 
its union is an ideal and $I \subset (\cup_{m=0}^\infty I_{a^m_\infty})_{a_\infty} = \cup_{m=0}^\infty I_{a^m_\infty}$.
If furthermore $I \subset st_{a_\infty}(\mathscr L_1)$, then also $I_{a_\infty}^m \subset st_{a_\infty}(\mathscr L_1) \subset \mathscr L_1$ for all $m$, hence 
$\cup_{m=0}^\infty I_{a^m_\infty} \subset \mathscr L_1$ and by Corollary \ref{C: a-infty stability equivalences}, $st_{a_\infty}(\mathscr L_1)$ is am-$\infty$ stable.
Notice that if $I \subset se(\omega)$ but $I \not\subset st_{a_\infty}(\mathscr L_1)$,
i.e., $I_{a^m_\infty} \not \subset \mathscr L_1$ for some $m\geq 0$, 
then, by Lemma \ref{L:(xi) subset (xiainfty)},
$I_{a^{m+1}_\infty} = se(\omega)$ and so $\bigcup_{m=0}^\infty I_{a^m_\infty} = se(\omega)$ which is not am-$\infty$ stable.

Thus, in particular, $st^a(\mathscr L_1)=st^a(F)=st^a((\omega))$ is the smallest am-stable ideal and 
$st_{a_\infty}(\mathscr L_1)=st_{a_\infty}(K(H))=st_{a_\infty}((\omega))$ is the largest am-$\infty$ stable ideal. 

\begin{remark}\label{R: p.i. towers}
\item[(i)] If $I$ is a principal ideal which is not am-stable, then $st^a(I)$ is a strictly increasing nested union of principal ideals.
Indeed, $(\xi_{a^n}) = I_{a^n} = I_{a^{n+1}} = (\xi_{a^{n+1}})$ if $I=(\xi)$ implying that $\xi_{a^n}$ is regular and hence as recalled before 
Theorem \ref{T: Theorem 4.12}, $\xi$ is regular, i.e., $(\xi)$ is am-stable. 
\item[(ii)] Similarly, if $I \subset st_{a_\infty}(\mathscr L_1)$ is principal and not 
am-$\infty$ stable, by Theorem \ref{T: Theorem 4.12} (the equivalence of (i) and (iv)) and Lemma \ref{L:(xi) subset (xiainfty)}, 
$st^{a_\infty}(I)$ is also a strictly increasing nested union of principal ideals. 
This phenomenon does not even extend to countably generated ideals as \cite[Example 5.5]{vKgW07-Density} shows by constructing a countably generated ideal $L \subsetneq L_a = L_{a^2}$.
\end{remark}

The next proposition shows that $st^a(\mathscr L_1)$ is the union of principal ideals and that $st_{a_\infty}(\mathscr L_1)$ is the intersection of Lorentz ideals. 
Recall from \cite[Sections 2.25, 2.27, 4.7]{DFWW} that if $\pi$ is a positive nondecreasing $\Delta_2$-sequence, 
then $\mathscr L(\sigma(\pi))$ is the Lorentz ideal with characteristic set 
$\Sigma(\mathscr L(\sigma(\pi))) := \{\xi \in \text{c}_{\text{o}}^{*} \mid \underset{n}\sum~ \xi_n\pi_n < \infty \}$.

\begin{proposition}\label{P: stabilizers with logs} 
\[
(i) \quad st^a(\mathscr L_1) = \bigcup_{m=0}^\infty (\omega ~log^m)
\]
\[
(ii) \quad st_{a_\infty}(\mathscr L_1) = \bigcap_{m=0}^\infty \mathscr L(\sigma(log^m))
\]
\end{proposition}

\begin{proof}
\item[(i)]  This is clear since $st^a(\mathscr L_1)=st^a((\mathscr L_1)_a) = st^a((\omega))$ and  $(\omega)_{a^m} = (\omega_{a^m}) = (\omega~log^m)$ for every $m \in \mathbb N$.
\item[(ii)] $\xi \in \Sigma(_{a_\infty}\mathscr L(\sigma(log^m)))$ if and only if $\xi \in (\ell^1)^*$ and 
\[
\sum_{n=1}^{\infty} ~(\sum_{j=n+1}^{\infty} \xi_j)~\frac{log^m n}{n} < \infty ~\text{if and only if}~\sum_{n=1}^{\infty} \xi_n ~log^{m+1} n < \infty,
\]
i.e., $\xi \in \Sigma(\mathscr L(\sigma(log^{m+1})))$. 
Therefore $_{a_\infty}\mathscr L(\sigma(log^m)) = \mathscr L(\sigma(log^{m+1}))$ and hence $_{a^m_\infty}(\mathscr L_1) = \mathscr L(\sigma(log^m))$.
\end{proof}

\noindent Thus, if $\xi \in \text{c}_{\text{o}}^*$ is $\infty$-regular, then $\xi \in \Sigma(st_{a_\infty}(\mathscr L_1))$ and hence  
$\sum_{n=1}^{\infty} \xi_n\,log^m n < \infty$ for every $m$ (cf. Example \ref{E:infty regulars/irregulars}(iv)).
Notice also that the proof of (ii) shows in particular that $_{a_\infty}(\mathscr L_1) = \mathscr L(\sigma(log))$.
In \cite{vKgW07-soft} we prove that $st^a(\mathscr L_1)$ and $st_{a_\infty}(\mathscr L_1)$ are both soft-edged and soft-complemented.

In Section 7 we will need the following analogues of Lemma \ref{L:se,sc,pre-am,am} and Proposition \ref{P:TFAE-seI am-stable, scI am-stable, seI in Ipre-a, Ia in scI}. 

\begin{lemma}\label{L: sc(ainftyI),(seI)ainfty} Let $I$ be an ideal.
\item[(i)] $sc(_{a_\infty}I) = ~_{a_\infty}(scI)$ when $I \ne \{0\}$
\item[(ii)] $(seI)_{a_\infty} = se(I_{a_\infty})$
\end{lemma}

\begin{proof}
\item[(i)] Let $\eta \in \Sigma(sc(_{a_\infty}I))$. Since $_{a_\infty}I \subset \mathscr L_1$, 
$sc(_{a_\infty}I) \subset sc\mathscr L_1 = \mathscr L_1$ and hence $\eta \in (\ell^1)^*$. 
Choose a strictly increasing sequence of positive integers $n_k$ with $n_{-1} = n_0 = 0$ for which 
$\sum_{n_k+1}^{n_{k+1}} \eta_j \geq \frac{1}{2} \sum_{n_k+1}^{\infty} \eta_j$ for $k \geq 0$. 

For each $\alpha \in \text{c}_{\text{o}}^*$ set  $\alpha_0 = \alpha_1$ and define 
\[
\widetilde \alpha_j := 2\alpha_{n_{k-1}} \quad \text{for $n_k < j \leq n_{k+1}$ and $k \geq 0$}.
\]
Then $\widetilde \alpha = ~<\widetilde \alpha_j>~ \in \text{c}_{\text{o}}^*$ and for all $n_k < p \leq n_{k+1}$ and $k \geq 0$,  
\begin{align*}
\sum_{p}^{\infty} \widetilde \alpha_j \eta_j &\geq \sum_{p}^{n_{k+2}} \widetilde \alpha_j \eta_j 
= 2\alpha_{n_{k-1}} \sum_{p}^{n_{k+1}} \eta_j + 2\alpha_{n_k} \sum_{n_{k+1}+1}^{n_{k+2}} \eta_j \\
&\geq \alpha_{n_{k-1}} \sum_{p}^{n_{k+1}} \eta_j + \alpha_{n_k} \sum_{n_{k+1}+1}^{\infty} \eta_j 
\geq \alpha_{n_k} \sum_{p}^{\infty} \eta_j \geq \alpha_p \sum_{p}^{\infty} \eta_j.
\end{align*} 

\noindent Thus $\alpha \eta_{a_\infty} \leq (\widetilde \alpha \eta)_{a_\infty}$. 
Since $\widetilde \alpha \eta \in \Sigma(_{a_\infty}I)$ from the  definition of $sc$ it follows that $(\widetilde \alpha \eta)_{a_\infty} \in \Sigma(I)$
and hence that $\eta_{a_\infty} \in \Sigma(scI)$, i.e., 
$\eta \in \Sigma(_{a_\infty}(scI))$. 
Hence $sc(_{a_\infty}I)) \subset~_{a_\infty}(scI)$.

Now let $\eta \in \Sigma(_{a_\infty}(scI))$ and $\alpha \in \text{c}_{\text{o}}^*$. 
Then $\eta_{a_\infty} \in \Sigma(scI)$ and hence $\alpha \eta_{a_\infty} \in \Sigma(I)$. \linebreak
Since 
$(\alpha \eta)_{a_\infty} \leq \alpha \eta_{a_\infty}$, 
also $(\alpha \eta)_{a_\infty}  \in \Sigma(I)$, i.e.,  
$\alpha \eta \in \Sigma(_{a_\infty}I)$ so $\eta \in \Sigma(sc(_{a_\infty}I))$, which yields the set equality.

\item[(ii)] Assume first that $I \not\subset \mathscr L_1$. 
Since $\mathscr L_1$ is soft-complemented, also $seI \not\subset \mathscr L_1$ since otherwise 
$I \subset scI = sc~seI \subset sc\mathscr L_1 = \mathscr L_1$. 
But then, from Lemma \ref{L:(xi) subset (xiainfty)}, it follows that $(seI)_{a_\infty} = se(\omega)$ and  
$I_{a_\infty} = se(\omega)$, hence $se(I_{a_\infty}) = se(\omega) = (seI)_{a_\infty}$.

Assume now that $I \subset \mathscr L_1$. 
If $\xi \in \Sigma((seI)_{a_\infty})$ then $\xi \leq \rho_{a_\infty}$ for some $\rho \in \Sigma(seI)$, 
i.e., $\rho \leq \alpha \eta$ for some $\alpha \in \text{c}_{\text{o}}^*$ and $\eta \in \Sigma(I)$. 
But then $\xi \leq (\alpha \eta)_{a_\infty} \leq \alpha \eta_{a_\infty}$. 
By definition, 
$\alpha \eta_{a_\infty} \in \Sigma(se(I_{a_\infty}))$, hence $(seI)_{a_\infty} \subset se(I_{a_\infty})$.
For the reverse inclusion, let $\xi \in \Sigma(se(I_{a_\infty}))$ and hence $\xi \leq \alpha \rho$ 
for some $\alpha \in \text{c}_{\text{o}}^*$ and $\rho \in \Sigma(I_{a_\infty})$, that is, $\rho \leq \eta_{a_\infty}$ for some $\eta \in \Sigma(I)$. 
As in the proof of part (i),  
$\xi \leq \alpha \eta_{a_\infty} \leq (\widetilde \alpha \eta)_{a_\infty}$
for some $\widetilde \alpha \in \text{c}_{\text{o}}^*$. 
But $\widetilde \alpha \eta \in \Sigma(seI)$ and therefore 
$\xi \in \Sigma((seI)_{a_\infty})$, which yields the set equality. 
\end{proof}

\begin{proposition}\label{P: TFAE:seI am-stable, scI am-stable, seI in, scI contains}
The following are equivalent for ideals $I \ne \{0\}$.
\item[(i)] $seI$ is am-$\infty$ stable.
\item[(ii)] $scI$ is am-$\infty$ stable.
\item[(iii)] $seI \subset ~_{a_\infty}I$
\item[(iii$'$)] $\omega \notin \Sigma(I)$ and $seI \subset F +[I, B(H)]$.
\item[(iv)] $\omega \notin \Sigma(scI)$ and $scI \supset I_{a_\infty}$
\end{proposition}

\begin{proof} 
\item[(i) $\Rightarrow$ (iii)] Since $seI \subset I$ and the pre-arithmetic mean at infinity is inclusion preserving, 
$seI = ~_{a_\infty}(seI) \subset~_{a_\infty}I$.

\item[(iii) $\Rightarrow$ (ii)]  
$scI = sc~seI \subset sc(_{a_\infty}I) =~ _{a_\infty}(scI) \subset scI$ 
where the first equality is true for all ideals (recall comment preceding Remark \ref{R:soft terminology}), the second equality follows from Lemma \ref{L: sc(ainftyI),(seI)ainfty}(i) and the last inclusion from Proposition \ref{P:aminfty relations}(i).

\item[(ii) $\Rightarrow$ (iv)] By Corollary \ref{C: a-infty stability equivalences}, $scI \subset \mathscr L_1$ so $\omega \notin \Sigma(scI)$. 
Also $scI=(scI)_{a_\infty} \supset I_{a_\infty}$.

\item[(iv) $\Rightarrow$ (i)]
$I \subset \mathscr L_1$ because otherwise
$se(\omega) = I_{a_\infty} \subset scI$ by Lemma \ref{L:(xi) subset (xiainfty)} and thus $(\omega) = sc(se(\omega))\subset scI$, against the hypothesis.
But then $seI \subset \mathscr L_1$ and hence $seI \subset (seI)_{a_\infty}$ from Proposition \ref{P:aminfty relations}(i$'$).
For the reverse inclusion, by Lemma \ref{L: sc(ainftyI),(seI)ainfty}(ii), 
$(seI)_{a_\infty} = se(I_{a_\infty}) \subset se(scI) = seI$, and so $(seI)_{a_\infty} = seI$. 
The conclusion now follows from Corollary \ref{C: a-infty stability equivalences}.
\item[(iii) $\Rightarrow$ (iii$'$)] Since (iii) implies (iv), $\omega \notin \Sigma(I)$ and by Corollary \ref{C:(F+[])+}(i)
    $seI \subset~_{a_\infty}I = span~ (_{a_\infty}I)^+ = span\, (F +[I, B(H)])^+ \subset F +[I, B(H)]$.
\item[(iii$'$) $\Rightarrow$ (iii)] Again by Corollary \ref{C:(F+[])+}(i), $(seI)^+ \subset (F +[I, B(H)])^+ =~(_{a_\infty}I)^+$, hence $seI\subset~ _{a_\infty}I$.
\end{proof}

\begin{remark}\label{R: I am-infty stable implies seI,scI am-infty stable}
Analogous to Remark \ref{R: I am-stable implies seI,scI am-stable with false converse}, 
the $se$ and $sc$ operations preserve am-$\infty$ stability by Lemma \ref{L: sc(ainftyI),(seI)ainfty}(i) and Proposition \ref{P: TFAE:seI am-stable, scI am-stable, seI in, scI contains}.
However, Example \ref{E:seI,scI infty stable but I not} below shows that $seI$ and $scI$ can be am-$\infty$ stable while $I$ is not.
\end{remark}
\newpage
\begin{example}\label{E:seI,scI infty stable but I not}
We construct an ideal I such that $seI = se(\omega^2)$ and hence $scI = (\omega^2)$ are am-$\infty$ stable but $I$ is not am-$\infty$ stable. 
\end{example}

\noindent Indeed, let $m_k = (k!)^2$ and define $\eta_j = \frac{1}{m_k^2}$  for $m_{k-1} < j \leq m_k$. 
Then $\eta \in \text{c}_\text {o}^*$, $\eta \leq \omega^2$ but $\eta \ne o(\omega^2)$. 
Set $I =  se(\omega^2) + (\eta)$. Then $seI = se(\omega^2)$ and hence $scI = (\omega^2)$. 
By Example \ref{E:infty regulars/irregulars}(i), $\omega^2$ is $\infty$-regular, i.e., $(\omega^2)$ is am-$\infty$ stable 
and thus so are $seI$ and $scI$. 
However $I$ is not am-$\infty$ stable. 
To prove this by contradiction, assume that it is. 
Then $\eta_{a_\infty} \in \Sigma(I)$, i.e., there is an $\alpha \in \text{c}_\text{o}^*, M > 0$, 
and $p \in \mathbb N$ such that $\eta_{a_\infty} \leq \alpha \omega^2 + MD_p\eta$. 
Without loss of generality assume that $\alpha_j = \epsilon_k$ for $m_{k-1} < j \leq m_k$. 
For every $m_{k-1} < n \leq m_k$,
\[
(\eta_{a_\infty})_n > \frac{1}{n} \sum_{j=n+1}^{m_k} \frac{1}{m_k^2} = \frac{m_k-n}{nm_k^2}.
\]
In particular, for $k > p$ by choosing $n = km_{k-1} = \frac{m_k}{k}$ we have
\[
\frac{k-1}{m_k^2} < ~\epsilon_k \frac{k^2}{m_k^2}+M(D_p\eta)_{km_{k-1}} = ~\epsilon_k \frac{k^2}{m_k^2}+M\frac{1}{m_k^2}.
\]
This implies $\epsilon_k > \frac{k-M-1}{k^2}>\frac{1}{2k}$ for $k$ large enough, in which case then $2\epsilon_k m_k > km_{k-1}$. 
Now by choosing $n = [2\epsilon_k m_k]$ and $k$ large enough to insure that $\epsilon_k \leq \frac{1}{2}$, 
we have $km_{k-1} \leq n \leq m_k$ and hence 
\[
\frac{m_k-n}{nm_k^2} < (\eta_{a_\infty})_n \leq \frac{\epsilon_k}{n^2}+(D_p\eta)_n = \frac{\epsilon_k}{n^2}+\frac{M}{m_k^2}.
\]
Then 
\[
M+1 \geq \frac{m_k}{n} - \frac{\epsilon_k m_k^2}{n^2} \geq  \frac{1}{2\epsilon_k} - \frac{\epsilon_k m_k^2}{(2\epsilon_k m_k-1)^2} 
= \frac{1}{\epsilon_k} \left(\frac{1}{2}-\frac{1}{(2-\frac{1}{\epsilon_km_k})^2}\right)
\]
which is a contradiction since $\epsilon_k m_k \rightarrow \infty$ and $\epsilon_k \rightarrow 0$.

\section{\leftline{\bf Nonsingular traces and applications to elementary operators}} 

It is well-known that 
the restriction of a trace on an ideal $I$ to the ideal $F$ of finite rank operators
must be a (possibly zero) scalar multiple of the standard trace $Tr$.

\begin{definition}\label{D:singular trace} A trace on an ideal that vanishes on $F$ is called 
singular, and nonsingular otherwise.
\end{definition}

Dixmier \cite{jD66} provided the first example of a (positive) singular trace. 
Its domain is the 
am-closure, $(\eta)^-=~_a(\eta_a)$, of a principal ideal 
$(\eta) \subset se(\eta)_a$.

Theorem \ref{T:DFWW} yields a complete characterization of ideals that support a nonsingular trace, 
namely, those ideals that do not contain $diag~\omega$ 
(cf. \cite[Introduction, Application 3 of Theorem 5.6]{DFWW} and also \cite{kDgWmW00}). 
For the reader's convenience this argument is presented and generalized in Proposition \ref{P:trace extension} below.
To prove it we first need another simple consequence of Theorem \ref{T:DFWW}. 

\begin{lemma}\label{L:nonsingular extension criteria} Let $I \ne \{0\}$ be an ideal for which $\omega \notin \Sigma(I)$.
\item[(i)] $\mathscr L_1 \cap [I,B(H)] \subset \{X\in\mathscr L_1\mid Tr~X=0\}$
\item[(ii)] $dim~\frac{F + [I,B(H)]}{[I,B(H)]} = 1$ 
\end{lemma}

\newpage

\begin{proof}
\item[(i)] 
Since the subspace $\mathscr L_1 \cap [I,B(H)]$ is the span of its selfadjoint 
elements, \linebreak
consider $X=X^* \in \mathscr L_1 \cap [I,B(H)]$. 
Then $\lambda(X)_a \in S(I)$ by Theorem \ref{T:DFWW}.
Then if $\sum_{1}^{\infty} \lambda(X)_j = Tr~X \neq 0$, it would follow that $|\lambda(X)_a| \asymp \omega$ and hence, 
by the hereditariness of $I^+$, $\omega \in \Sigma(I)$, against the hypothesis.
\item[(ii)] Fix a rank one projection $P$. 
Then $\lambda(P)_a = \omega$ and so $P \notin [I,B(H)]$ by Theorem \ref{T:DFWW}.
For each $X \in F + [I,B(H)]$ choose $T \in F$ for which $X-T \in [I,B(H)]$. 
As $T - (Tr~T)P \in [F,B(H)] \subset [I,B(H)]$ from the well-known fact that a finite complex matrix is a commutator if and only if it has zero trace (cf. \cite[Discussion of Problem 230]{pH82}), 
one obtains $X - (Tr~T)P \in [I,B(H)]$. 
Thus
\[F + [I,B(H)] = \{\lambda P + [I,B(H)] \mid \lambda \in \mathbb C\}.\]
\end{proof}

\begin{proposition}\label{P:trace extension} Let $I$ and $J$ be ideals and let $\tau$ be a trace on $J$. 
Then $\tau$ has a trace extension to $I+J$ if and only if 
\[
 J\cap[I,B(H)]\subset \{\,X \in J\mid\tau(X)=0\,\}.
\]
Moreover, the extension is unique if and only if $I \subset J + [I,B(H)]$.

In particular, $\omega\notin\Sigma(I)$ 
if and only if $Tr$ extends from $\mathscr L_1$ to $\mathscr L_1+\,I$ if and only if $Tr$ extends from $F$ to $I$.
\end{proposition}

\begin{proof}
Assume that a trace $\tau$ on $J$ has a trace extension $\widetilde{\tau}$ to $I+J$ and that \linebreak
$X \in J\cap[I,B(H)]$.
Then $X \in [I+J,B(H)]$ and hence $\tau(X) = \widetilde{\tau}(X) = 0$ since 
every trace on $I+J$ 
must vanish on the commutator space of $I+J$. 

Conversely, assume that $J\cap[I,B(H)]\subset\{\,X \in J\mid\tau(X)=0\,\}$.
For $X \in J+[I,B(H)]$ choose $Y \in J$ for which 
$X-Y \in [I,B(H)]$ and
define $\tau'(X) := \tau(Y)$. 
As is easy to verify, $\tau'$ is a well-defined linear functional on $J+[I,B(H)]$,
it extends $\tau$, it vanishes on $[I,B(H)]$ and hence also on $[I+J,B(H)]=[I,B(H)]+[J,B(H)]$, and it is the unique linear extension of $\tau$ to $J+[I,B(H)]$ that vanishes on $[I,B(H)]$. \linebreak
If $I+J \neq J+[I,B(H)]$, use a Hamel basis argument to further extend $\tau'$ to a linear functional $\tau''$ on $I+J$, 
and this extension is never unique.
Since $\tau''$ vanishes on $[I+J,B(H)]$ it is a trace on $I+J$.

The case for extending $Tr$ from $\mathscr L_1$ or $F$, when $\omega \notin \Sigma(I)$, follows from Lemma \ref{L:nonsingular extension criteria}(i).
Conversely, if $\omega\in \Sigma(I)$, then $\mathscr L_1 \subset I$ and by Theorem \ref{T:DFWW}, $\mathscr L_1^+ \subset [I,B(H)]$, whence $\mathscr L_1 \subset [I,B(H)]$. Thus all traces on $I$ vanish on $\mathscr L_1$, i.e., $Tr$ cannot extend from $F$ or $\mathscr L_1$ to $I$.
\end{proof}

A simple consequence of this is that all traces on $J$ extend to $I+J$ if and only if $J\cap[I,B(H)]\subset [J,B(H)]$ since, as is elementary to show, 
\[[J,B(H)] =  \bigcap \{\{\,X \in J\mid\tau(X)=0\,\} \mid \tau ~\text{is a trace on}~ J\}.\]

\begin{remark}\label{R:extension,=,L1extensions cant be pos when omega not in I}
\item[(i)] A routine argument shows that for any ideal $J$ and trace $\tau$ on $J$, the collection of the ideals to which $\tau$ can be extended, 
i.e. the ideals $I\supset J$ for which \begin{center}$J\cap  [I, B(H)] \subset \{X \in J | \tau(X) = 0\}$, \end{center}
is closed under directed unions and hence it always has maximal elements. 
Because this collection is hereditary with respect to inclusion, it is closed under addition if and only if it has a unique maximal element.
That this may not be the case is easy to show for $\tau = Tr$ and $J = F$ by constructing two principal ideals 
$I_1$ and $I_2$ with $I_1\ne I_2$, $I_1 + I_2 = (\omega)$, but $I_i \ne (\omega)$ for $i=1,2$.
\item[(ii)] 
If a trace $\tau$ on and ideal $J$ has extensions $\tau_i$ to the ideals $I_1 \subset I_2$, there is no reason for $\tau_1$ to have an extension to $I_2$. 
For instance, again in the case of $\tau= Tr$, $J = F$, and $\omega  \notin \Sigma(I)$, 
there is a unique trace on $I$ if and only if $I$ is am-$\infty$ stable (see Theorem \ref{T:i-iii} below). 
Thus if $I_1$ is not am-$\infty$ stable but is contained in an am-$\infty$ stable ideal $I_2$,  
then all but one of the traces on $I_1$ do not further extend to $I_2$.

\item[(iii)] By Lemma \ref{L:nonsingular extension criteria}(i), the set equality holds always in 
\begin{center}$J\cap  [I, B(H)] \subset \{X \in J | \tau(X) = 0\}$\end{center} 
for $\tau = Tr$, $J = F$ and any ideal $I \not\supset (\omega)$. 
By Proposition \ref{P:omega notin I, codim, TFAE} below 
(see also the remark following it), 
equality holds for $\tau = Tr$, $J = \mathscr L_1$ and an ideal $I\supset \mathscr L_1$ to which $Tr$ can be extended (i.e., $\omega  \notin \Sigma(I)$)
if and only if $se(\omega) \subset I$. 
Notice that if $Tr$ is extendable to $I$, then $Tr$ is extendable also to $se(\omega) + I$ as $\omega  \notin \Sigma(se(\omega) + I)$. 
So every maximal ideal $I$ for the extension must contain $se(\omega)$ and satisfies the set equality.

\item[(iv)] Although $Tr$ is positive on $\mathscr L_1$ and it has extensions to any ideal $I$ properly containing $\mathscr L_1$ but not containing $(\omega)$(uncountably many according to Corollary \ref{C:dimensions of 3 quotients}(iii)), none of these extensions can be positive. 
Indeed, as is well known (e.g., see the proof of \cite[Lemma 2.15]{mW02} or \cite[Remark 2.5]{AGPS87}), $\tau \geq \tau (diag\,(1, 0, 0,...))\,Tr$, for any positive trace $\tau$. 
Thus if $I$ properly contains $\mathscr L_1$, it follows that $\tau (diag\,(1, 0, 0,...)) = 0$, so $\tau$ is singular, i.e., $\tau$ does not extend $Tr$. 
\end{remark}

The following result will also be useful.

\begin{proposition}\label{P:(L1+[])+=L1+} Let $I$ be an ideal for which $\omega \notin \Sigma(seI)$. Then 
\[
(\mathscr L_1 + [I,B(H)])^+ = \mathscr L_1^+.
\]
\end{proposition}

\begin{proof} Let $0 \leq X \in \mathscr{L}_1+[I,B(H)]$ and so $X-T\in [I,B(H)]$ for some $T \in \mathscr{L}_1$.
Since $X=X^*$ and $[I,B(H)]=[I,B(H)]^*$, assume without loss of generality that $T=T^*$. 
Let $f_j$ be an orthonormal basis for $\mathscr N(X-T)$, the null space of $X-T$, \linebreak
and let $e_j$ be an orthonormal basis of eigenvectors of $X-T$ for 
$\mathscr N(X-T)^\perp$ arranged so that  
$|((X-T)e_j,e_j)|$ is monotone nonincreasing. 
Since $\sum (Xf_j,f_j) = \sum (Tf_j,f_j) < \infty$, 
to prove that $X \in \mathscr L_1$ it suffices to show that $\sum (Xe_j,e_j) < \infty$.
But if otherwise $\sum_{1}^{\infty} (Xe_j,e_j) = \infty$, then  
$\sum_{1}^{\infty} ((X-T)e_j,e_j) = \infty$ 
since $T \in \mathscr{L}_1$, \linebreak
and so $\omega = o(\frac{\sum_{j=1}^{n} ((X-T)e_j,e_j)}{n})$.
By Theorem~\ref{T:DFWW}, 
$\langle |\frac {\sum_{j=1}^{n} ((X-T)e_j,e_j)}{n}| \rangle \leq \xi$ for some $\xi \in \Sigma(I)$
and hence $\omega = o(\xi)$,
against the hypothesis.
\end{proof}

\begin{remark}\label{R:(L1+[])+=L1+} 
\item[(i)] The condition $\omega \notin \Sigma(seI)$ is also necessary since in \cite[Lemma 6.2(iv)]{vKgW07-soft} we show that if 
$\omega \in \Sigma(seI)$, then $_aI \supsetneq \mathscr L_1$ and therefore $[I,B(H)]^+ = (_aI)^+ \supsetneq \mathscr L_1^+$.
\item[(ii)] 
If $\omega \notin \Sigma(I)$ and $I \supset \mathscr L_1$, then the extension of $Tr$ to $I$ is unique only in the trivial case $I = \mathscr L_1$.
Indeed, from Proposition \ref{P:trace extension}, uniqueness implies $I \subset \mathscr L_1 + [I,B(H)]$, 
which implies $I^+ \subset \mathscr L_1^+$ by Proposition \ref{P:(L1+[])+=L1+}. 
Corollary \ref{C:dimensions of 3 quotients}(iii) will show that, but for the $I = \mathscr L_1$ case, there are always uncountably many linearly independent extensions of $Tr$ to $I$. 
\end{remark}

Trace extensions find natural applications to questions on \emph{elementary operators}. 
If $A_i, B_i \in B(H)$, then the $B(H)$-map 
\[
B(H) \owns T \rightarrow \Delta(T) := \sum_{i=1}^{n} A_iTB_i
\]
is called an elementary operator and its adjoint $B(H)$-map is 
$\Delta^*(T) := \sum_{i=1}^{n} A_i^*TB_i^*$. 
Elementary operators include commutators and intertwiners and hence their theory is connected to the structure of commutator spaces. 
The Fuglede-Putnam Theorem \cite{bF51}, \cite{cP51} states that for the case $\Delta(T) = AT-TB$ where A, B are normal operators, 
$\Delta(T) = 0$ implies that $\Delta^*(T) = 0$. 
Also for $n = 2$, Weiss \cite{gW83} 
generalized this further to the case where \{$A_i$\} and \{$B_i$\}, $i=1,2$, are 
separately commuting families of normal operators by proving that 
$\Delta(T) \in \mathscr{L}_2$ implies $\Delta^*(T) \in \mathscr{L}_2$ and 
$\Vert\Delta(T)\Vert_2 = \Vert\Delta^*(T)\Vert_2$. 
(This is also a consequence of Voiculescu's 
\cite[Theorem 4.2 and Introduction to Section 4]{dV79/81} but neither Weiss' nor Voiculescu's methods seem to apply to the case $n > 2$.) 
In \cite{vS83} Shulman showed that for $n = 6$, $\Delta(T) = 0$ does not imply $\Delta^*(T) \in \mathscr{L}_2$.

If we impose some additional conditions involving ideals on the families $\{A_i\},\{B_i\}$ and $T$,  
we can extend these implications to arbitrary $n$ past the obstruction found by Shulman. 

Assume that \{$A_i$\}, \{$B_i$\}, $i = 1,\dots,n$, 
are separately commuting families of normal operators and let $T \in B(H)$. 

Define the following ideals: 
\begin{align*}
&L := (\sum_{i=1}^{n}(A_i T)(B_i))^2, \quad L_* := (\sum_{i=1}^{n}(A_i^* T)(B_i))^2, \\
&R := (\sum_{i=1}^{n}(A_i)(TB_i))^2, \quad R_* := (\sum_{i=1}^{n}(A_i)(TB_i^*))^2 ~\text{and}\\
&I_{\Delta,T} := L \cap L_* \cap R \cap R_*, \quad S = (\sum_{i=1}^{n}(A_i T B_i)) \cap L^{1/2}\cap R^{1/2}
\end{align*}
where $(X)$ denotes the principal ideal generated by the operator $X$ and $MN$ (resp., $M+N$) denotes the product (resp., sum) of the ideals $M,N$. 
Then $I_{\Delta,T}$ and $S$ are either $\{0\},~B(H)$, or a principal ideal. 

\begin{proposition}\label{P:FugL2} 
If $\omega \notin \Sigma(I_{\Delta,T})$, then 
$\Delta(T) \in \mathscr{L}_2$ implies \begin{center}$\Delta^*(T) \in \mathscr L_2$ and
$\Vert\Delta(T)\Vert_2 = \Vert\Delta^*(T)\Vert_2$.\end{center}
\end{proposition}

\begin{proof}
Define $I_1=L \cap L_* \cap (R + R_*)$ and $I_2=R \cap R_* \cap (L + L_*)$ so $I_{\Delta,T}=I_1\cap I_2$. Assume first that $\omega \notin \Sigma(I_1)$.
We start by showing that $|\Delta^*(T)|^2 - |\Delta(T)|^2 \in [L,B(H)]$. 
Observe that 
\[
|\Delta(T)|^2 = \sum_{i,j=1}^{n} B_j^*T^*A_j^*A_iTB_i \quad \text{and} \quad |\Delta^*(T)|^2 = \sum_{i,j=1}^{n} B_iT^*A_iA_j^*TB_j^*.
\] 
Then for each $i,j$, 
\begin{align*}
B_j^*T^*A_j^*A_iTB_i-T^*A_j^*A_iTB_iB_j^* &\in  [(B_j^*),(T^*A_j^*)(A_iT)(B_i)] \\
&= [(B_j),(A_jT)(A_iT)(B_i)]  \\
&= [(A_jT)(B_j)(A_iT)(B_i),B(H)] \subset  [L,B(H)]. 
\end{align*}
Here use the elementary facts that $(X) = (X^*)$ for every operator $X$, 
that the product of ideals is a commutative operation, and use the 
deep identity $[M,N]=[MN,B(H)]$ \cite[Theorem 5.10]{DFWW} for ideals $M,N$. 
By the Fuglede-Putnam Theorem \cite{cP51} and the assumption that \{$B_i$\} 
are normal and commuting we get that, for all $i,j$: 
\[
T^*A_j^*A_iTB_iB_j^* = T^*A_j^*A_iTB_j^*B_i.
\]
Then, as above,  
\[
T^*A_j^*A_iTB_j^*B_i - B_iT^*A_j^*A_iTB_j^* \in [L,B(H)].
\]
Again by the Fuglede-Putnam Theorem, $B_iT^*A_j^*A_iTB_j^* = B_iT^*A_iA_j^*TB_j^*$, 
which proves the claim. 
Interchanging the role of $\Delta$ and $\Delta^*$ yields also 
$|\Delta^*(T)|^2 - |\Delta(T)|^2 \in [L_*,B(H)]$.
On the other hand, by the same argument, for all $i,j$ one has
\[
B_j^*T^*A_j^*A_iTB_i-A_j^*A_iTB_iB_j^*T^* \in [R,B(H)]
\]
and by applying twice the Fuglede-Putnam Theorem,
\[
A_j^*A_iTB_iB_j^*T^* = A_iA_j^*TB_j^*B_iT^*.
\]
Similarly, \quad
$
A_iA_j^*TB_j^*B_iT^* - B_iT^*A_iA_j^*TB_j^* \in [R_*,B(H)]
$
and hence 
\[|\Delta^*(T)|^2 - |\Delta(T)|^2 \in [R,B(H)]+[R_*,B(H)] = [R+R_*,B(H)].
\]
A simple consequence of Theorem \ref{T:DFWW} is that for ideals $M,N$, 
\[
[M,B(H)] \cap [N,B(H)] = [M \cap N,B(H)].
\]
Therefore $|\Delta^*(T)|^2 - |\Delta(T)|^2 \in [I_1,B(H)]$.

If $\Delta(T) \in \mathscr L_2$ and hence $|\Delta(T)|^2 \in \mathscr L_1$, 
then by Proposition \ref{P:(L1+[])+=L1+}, 
\begin{center}$|\Delta^*(T)|^2 \in (\mathscr L_1 + [I_1,B(H)])^+ = \mathscr L_1^+$.\end{center}
Moreover, by Proposition \ref{P:trace extension}, there is a trace extension $\tau$ of $Tr$ 
from $\mathscr L_1$ to $\mathscr L_1 + I_1$.
Since $\tau$ vanishes on $[I_1,B(H)] \subset \mathscr L_1 + [I_1,B(H)]$, so
$\tau(|\Delta(T)|^2) = \tau(|\Delta^*(T)|^2)$ and hence $Tr(|\Delta(T)|^2) = Tr(|\Delta^*(T)|^2)$. 

If $\omega \in \Sigma(I_1)$, then $\omega \notin \Sigma(I_2)$ and then we apply the same arguments to show that 
$\Delta^*(T)(\Delta^*(T))^* - \Delta(T)(\Delta(T))^* \in [I_2,B(H)]$ and to draw the same conclusions.
\end{proof}
\noindent A sufficient condition independent of $T$ that insures that $\omega \notin \Sigma(I_{\Delta,T})$ is 
\begin{center}$\omega^{1/4} \ne O(\sum_{i=1}^n(s(A_i)+s(B_i)))$. \end{center}
So also is the condition 
$\omega^{1/2} \ne O(\sum_{i=1}^n s(A_i))$ or the condition $\omega^{1/2} \ne O(\sum_{i=1}^n s(B_i))$.

\newpage

Propositions \ref{P:trace extension} and \ref{P:(L1+[])+=L1+} can also be applied to a problem of Shulman. Let 
\[
\Delta(T)=\sum_{i=1}^{n} A_iTB_i
\]
 be an elementary operator where the operators $A_i$ and $B_i$ are not assumed to be  commuting or normal. Shulman showed that the composition $\Delta^*(\Delta(T)) = 0$ does not imply $\Delta(T) = 0$ 
and conjectured that this implication holds under the additional 
assumption that $\Delta(T) \in \mathscr{L}_1$.  
In the case that the ideal $S$ is ``not too large" we can prove the implication without making this assumption.

\begin{proposition}\label{P:Shulman problem}
If $\omega \notin \Sigma(S)$, then $\Delta^*(\Delta(T)) \in \mathscr{L}_1$ implies that
$\Delta(T) \in \mathscr{L}_2$ 
and $\Vert\Delta(T)\Vert_2 = Tr~ T^*\Delta^*(\Delta(T))$. 

In particular, if $\Delta^*(\Delta(T)) = 0$ then $\Delta(T) = 0$.
\end{proposition}

\begin{proof}
Let $S_L = (\sum_{i=1}^{n}(A_i T B_i)) \cap L^{1/2}$ and $S_R  = (\sum_{i=1}^{n}(A_i T B_i)) \cap R^{1/2}$, so that $S =S_L \cap S_R$. 
Assume that $\omega \notin \Sigma(S_L)$.
Using the first step in the proof of Proposition \ref{P:FugL2} one has
\begin{align*}
|\Delta(T)|^2 - T^*\Delta^*(\Delta(T)) &= \sum_{i,j=1}^{n} (B_j^*T^*A_j^*A_iTB_i-T^*A_j^*A_iTB_iB_j^*) \\
&\in [S_L,B(H)].
\end{align*}
So if $\Delta^*(\Delta(T)) \in \mathscr{L}_1$ then 
$|\Delta(T)|^2 \in (\mathscr{L}_1 + [S_L,B(H)])^+ = \mathscr{L}_1^+$ by 
Proposition \ref{P:(L1+[])+=L1+}.
The required equality then follows by the same reasoning as in the conclusion of the proof of Proposition \ref{P:FugL2}.

If $\omega \in \Sigma(S_L)$, then $\omega \notin \Sigma(S_R)$ and we reach the same conclusions by considering $|\Delta(T)^*|^2- \Delta^*(\Delta(T))T^* \in [S_R, B(H)]$.
\end{proof}

\section{\leftline{\bf Uniqueness of Traces}}

An ideal $I$ supports a unique nonzero trace (up to scalar multiplication) 
precisely when $dim~\frac{I}{[I,B(H)]} = 1$.
In this section we characterize in terms of arithmetic means at infinity when this occurs for those ideals where $\omega \notin \Sigma(I)$. 

The next proposition is based on Theorem \ref{T:DFWW} which Kalton \cite{nK98} extended to non-normal operators 
for the class of geometrically stable ideals. These are the ideals $I$ for which $\Sigma(I)$ is invariant under geometric means,
that is,
\[
\xi \in \Sigma(I) ~\text{implies}~ \xi_g :=~ <(\xi_1\cdots\xi_n)^{\frac{1}{n}}> ~\in \Sigma(I).
\]

Notice that if $X \in I$ and $\lambda(X)$ and $\widetilde\lambda(X)$ are two different 
orderings of the sequence of all the eigenvalues (if any) of $X$, 
repeated according to algebraic multiplicity, augmented  
by adding infinitely many zeros when there are only a finite number of nonzero eigenvalues, 
and arranged so that both $|\lambda(X)|$ and $|\widetilde\lambda(X)|$ are monotone nonincreasing, 
then $|\lambda(X)|$ and $|\widetilde\lambda(X)| \in \Sigma(I)$ and it is elementary to show that 
$|\widetilde\lambda(X)_a| \leq |\lambda(X)_a| + 2|\lambda(X)|$. 
Similarly, $|\widetilde\lambda(X)_{a_\infty}| \leq |\lambda(X)_{a_\infty}| + 2|\lambda(X)|$ when \linebreak
$X \in \mathscr L_1 \cap I$. 
For this, notice that there is an increasing sequence of indices $n_k$ with $n_1=1$ 
for which $|\lambda(X)|_j=|\tilde \lambda(X)|_j=|\lambda(X)|_{n_k}$ for $n_k \leq j < n_{k+1}$.
Then $\sum_{n_k}^{n_{k+1}-1} \lambda(X)_j = \sum_{n_k}^{n_{k+1}-1} \tilde \lambda(X)_j$ for all $k$ and hence 
$\sum_{n_k}^{\infty} \lambda(X)_j = \sum_{n_k}^{\infty} \tilde \lambda(X)_j$. If $n_k \leq n < n_{k+1}$ then 
\begin{align*} 
|(\tilde\lambda(X)_{a_\infty})_n| - |(\lambda(X)_{a_\infty})_n| &\leq |(\tilde\lambda(X)_{a_\infty})_n - (\lambda(X)_{a_\infty})_n| \\
&= \frac{1}{n} |\sum_{n+1}^{\infty} \tilde \lambda(X)_j - \sum_{n+1}^{\infty} \lambda(X)_j| \\
& = \frac{1}{n} |\sum_{n_k}^n \tilde \lambda(X)_j - \sum_{n_k}^n \lambda(X)_j| \\
&\leq 2|\lambda(X)_n|.
\end{align*}
Thus $\lambda(X)_a \in S(I)$ (resp., $\lambda(X)_{a_\infty} \in S(I)$) if and only if 
$\widetilde\lambda(X)_a \in S(I)$ (resp., $\widetilde\lambda(X)_{a_\infty} \in S(I)$).
This illustrates in an elementary way why the choice of the ordering for $\lambda(X)$ does not matter in Theorem \ref{T:DFWW} and in Proposition \ref{P:F+[]} below. (See also \cite[Theorem 5.6]{DFWW}.)

\begin{proposition}\label{P:F+[]}
Let $I$ be an ideal, let $X \in \mathscr{L}_1 \cap I$, and assume that either $X$ is normal or $I$ is 
geometrically stable.
Then \begin{center}$X \in F + [I,B(H)]$ if and only if 
$\lambda(X)_{a_\infty} \in S(I)$. \end{center}
\end{proposition}
\begin{proof}
Assume first $\omega \in \Sigma(I)$ and so, by Theorem \ref{T:DFWW}, $F \subset [I,B(H)]$. 
Because $\lambda(X)_a + \lambda(X)_{a_\infty} = (Tr~X)\omega$, 
one sees that $\lambda(X)_{a_\infty} \in S(I)$ 
if and only if 
$\lambda(X)_a \in S(I)$, which, 
by Theorem \ref{T:DFWW} if $X$ is normal or by \cite{nK98} if $I$ is geometrically stable, is then equivalent to the condition $X \in [I,B(H)]=F+[I,B(H)]$.

Assume now that $\omega \notin \Sigma(I)$.
In case $X$ is quasinilpotent, 
i.e., 
$\lambda(X) = 0$, 
then $\lambda(X)_{a_\infty} = \lambda(X)_a = 0$ are in $S(I)$ 
and so, 
by 
Theorem \ref{T:DFWW}, if $X$ is normal or, by \cite{nK98}, 
if $I$ is geometrically stable, one has that $X \in [I,B(H)]$.
On the other hand, if $\lambda(X) \neq 0$, let $P$ be a rank one projection on an eigenvector of $X$ corresponding to the eigenvalue $\lambda(X)_1$. 
From the proof of Lemma \ref{L:nonsingular extension criteria}(ii) and by Lemma \ref{L:nonsingular extension criteria}(i), 
$X \in F+[I,B(H)]$ if and only if $Y := X - (Tr~X)P \in [I,B(H)]$. 
Now, by Theorem \ref{T:DFWW}, if $X$ and hence $Y$ are normal or, by \cite{nK98}, 
if $I$ is geometrically stable, 
$Y \in [I,B(H)]$ if and only if $\lambda(Y)_a \in S(I)$. 
$X$ can 
be represented as a $2 \times 2$ block matrix where the upper left block is upper triangular 
(and $\lambda(X)_1$ lies in its $(1,1)$ position) 
and the lower right block is quasinilpotent \cite[Proposition 2.1]{kDnK98}.
Also it is known that for compact upper triangular operators, the diagonal sequence is precisely the eigenvalue sequence repeated by algebraic multiplicity.
Therefore an eigenvalue sequence of $Y$ counting multiplicity is 
\[
<\lambda(X)_1-Tr~X,\lambda(X)_2,\lambda(X)_3,\dots>.
\]
Thus a monotonization in modulus of this sequence is given by
\[
\lambda(Y) = 
\begin{cases}
<\lambda(X)_2,\lambda(X)_3,\dots,\lambda(X)_p,\lambda(X)_1-Tr~X,\lambda(X)_{p+1},\dots> \\
\hspace{30pt} \text{for some $p\geq 1$ if $\lambda(X)_1 \neq Tr~X$} \\
<\lambda(X)_2,\lambda(X)_3,\dots> \indent \text{if $\lambda(X)_1 = Tr~X$}
\end{cases}
\]
where for $p=1$ we mean $\lambda(Y)=~ <\lambda(X)_1-Tr~X,\lambda(X)_2,\lambda(X)_3,\dots>$. Then
\[
\sum_{1}^{n}\lambda(Y)_j = 
\begin{cases}
\sum_{1}^{n}\lambda(X)_j - Tr~X = -\sum_{n+1}^{\infty}\lambda(X)_j &\text{if $\lambda(X)_1 \neq Tr~X$} \\
												&\text{for $n \geq p$} \\
\sum_{2}^{n+1}\lambda(X)_j = -\sum_{n+2}^{\infty}\lambda(X)_j	&\text{if $\lambda(X)_1 = Tr~X$} \\
												 &\text{for $n=1,2,\dots$.}
\end{cases}
\]
So, in either case, $\lambda(Y)_a \in S(I)$ if and only if $\lambda(X)_{a_\infty} \in S(I)$.
\end{proof}

Now we can link the cones of positive operators $(_{a_\infty}I)^+$ and $(_aI)^+$ to $(F + [I,B(H)])^+$.

\begin{corollary}\label{C:(F+[])+} Let $I\ne \{0\}$ be an ideal.
\item[(i)] If $\omega \notin \Sigma(I)$, then $(F + [I,B(H)])^+ = (_{a_\infty}I)^+$.
\item[(ii)] If $\omega \in \Sigma(I)$, then $(F + [I,B(H)])^+ = (_aI)^+$.
\item[(iii)] $(F + [I,B(H)])^+$ is hereditary (i.e., solid).
\end{corollary}

\begin{proof}
\item[(i)] By Proposition \ref{P:(L1+[])+=L1+}, 
$(F + [I,B(H)])^+ \subset \mathscr L_1^+$ and so by Proposition \ref{P:F+[]}, 
\begin{center}$X \in (F + [I,B(H)])^+$ if and only if $X \in (\mathscr L_1 \cap I)^+$ and $\lambda(X)_{a_\infty} \in \Sigma(I)$,\end{center} 
i.e., $X \in (_{a_\infty}I)^+$.
\item[(ii)] That $\omega \in \Sigma(I)$ implies $F \subset [I,B(H)]$
and $[I,B(H)]^+=(_aI)^+$ both
follow from 
Theorem \ref{T:DFWW} (see also its succeeding reformulation for positive operators).
\item[(iii)] This is immediate from (i) and (ii) since the positive cone of an ideal is hereditary.
\end{proof}

\noindent So, for instance, by combining (i) with the proof of Proposition \ref{P: stabilizers with logs}(ii) one obtains 
$(F+[\mathscr{L}_1,B(H)])^+ = (_{a_\infty}(\mathscr L_1))^+ = \mathscr L(\sigma(log))^+$ 
(the positive cone of a Lorentz ideal).

\begin{remark}\label{R: X,|X|,X oplus -X,(X)}
\item[(a)] 
In \cite[Theorem 5.11(i)]{DFWW} it is shown that $|X| \in [I, B(H)]$ if and only if \linebreak
$(X) \subset [I,B(H)]$. 
The proof depends on $[I,B(H)]^+ = (_aI)^+$ being hereditary. 
By the same argument combined with Corollary \ref{C:(F+[])+}(iii), it follows that 
\[
|X| \in F+[I,B(H)] \quad \text{if and only if}\quad (X) \subset F+[I,B(H)].
\]
\item[(b)] $X \in [I,B(H)]$ implies $|X| \in [I,B(H)]$ 
(resp., if $\omega \notin \Sigma(I)$, $X \in F+[I,B(H)]$ implies $|X| \in F+[I,B(H)]$) if and only if $I$ is am-stable (resp., $I$ is am-$\infty$ stable).
The condition is sufficent: if $I$ is am-stable (resp., am-$\infty$ stable), 
then $[I,B(H)]=I$ (resp., $F+[I,B(H)]=I$) is an ideal and for ideals, containment of $X$ and $|X|$ are equivalent. 
The condition is necessary: 
for every $Y \in I$ and in particular for every $Y \in I^+$,
\[
Y \oplus (-Y) ~=~
\begin{pmatrix}
0	&Y \\
0	&0
\end{pmatrix}
\begin{pmatrix}
0	&0 \\
I	&0
\end{pmatrix}
-
\begin{pmatrix}
0	&0 \\
I	&0
\end{pmatrix}
\begin{pmatrix}
0	&Y \\
0	&0
\end{pmatrix}
~\in [I,B(H)]
\]
where $H$ is identified with $H \oplus H$ and $I$ is identified with $M_2(I)$, 
the set of $2 \times 2$ matrices with entries in $I$. 
Thus by assumption 
$|Y \oplus (-Y)| = Y \oplus Y \in [I,B(H)]$, and $Y \oplus ~0 \in [I,B(H)]$ for every $Y \in I^+$ by hereditariness. 
Since every positive $X \in I$ is unitarily equivalent to $Y \oplus 0 + 0 \oplus Z$ for some $Y,Z \in I^+$, it follows that
$I^+ \subset [I,B(H)]^+ = (_aI)^+$. Thus $I = ~I_a$, i.e., $I$ is am-stable. 
The same argument shows that when $\omega \notin \Sigma (I)$, if $X \in F+[I,B(H)]$ implies $|X| \in F+[I,B(H)]$, 
then $I^+ \subset (F+[I,B(H)])^+ = (_{a_\infty}I)^+$ by Corollary \ref{C:(F+[])+}(i) and hence $I$ is am-$\infty$ stable.
\item[(c)] 
Notice that 
the same $2 \times 2$ matrix argument shows that $I$ is the smallest ideal containing $[I,B(H)]$.
\end{remark}

\begin{proposition}\label{P:omega notin I, codim, TFAE} For ideals $I$ and $J \ne \{0\}$ and arbitrary $J$, if $\omega \notin \Sigma(I)$,
then the following are equivalent.
\item(i) $dim~\frac{J+[I,B(H)]}{[I,B(H)]} = 1$
\item(ii) $J \subset F+[I,B(H)]$
\item(iii) $J \subset \mathscr L_1$ and $J \cap [I,B(H)] = \{X \in J \mid Tr~X=0\}$
\item(iv) $J \subset~_{a_\infty}I$
\item(iv$'$) $J \subset \mathscr L_1$ and $J_{a_\infty} \subset I$
\end{proposition}

\begin{proof}
\item[(i) $\Leftrightarrow$ (ii)] Immediate from Lemma \ref{L:nonsingular extension criteria}(ii) and the identity:
\[
dim~\frac{J + [I,B(H)]}{[I,B(H)]} = dim~\frac{J + [I,B(H)]}{F + [I,B(H)]} + dim~\frac{F + [I,B(H)]}{[I,B(H)]}.
\]
\item[(ii) $\Leftrightarrow$ (iv)] 
This follows from the equivalences $J \subset F + [I,B(H)]$ if and only if \linebreak
$J^+ \subset (F + [I,B(H)])^+ = (_{a_\infty}I)^+$ by Corollary \ref{C:(F+[])+}(i), 
if and only if $J \subset ~_{a_\infty}I$.  
\item[(iv) $\Leftrightarrow$ (iv$'$)] See Corollary \ref{P:aminfty relations}(iii). 
\item[(ii) $\Rightarrow$ (iii)] 
$J \subset \mathscr L_1$ since (ii) implies (iv$'$), hence
\[J \cap [I,B(H)] \subset \{X \in J \mid Tr~X = 0\}\] 
follows from 
Lemma \ref{L:nonsingular extension criteria}(i) without needing to invoke hypothesis (ii).
For the reverse inclusion, let $X \in J$ and $Tr~X = 0$. By (ii), $X - T \in [I,B(H)]$ for some $T \in F$, 
hence $Tr(X-T) = 0$ by the previous inclusion, and therefore $Tr~T = 0$. 
Then $T \in [F,B(H)] \subset [I,B(H)]$ as seen in the proof of Lemma \ref{L:nonsingular extension criteria}(ii), hence 
$X \in [I,B(H)]$ and thus (iii) holds.
\item[(iii) $\Rightarrow$ (ii)] Let $X \in J$ and let $P$ be a rank one projection.
But then \\ 
$Tr(X - (Tr~X)P)=0$, hence $X - (Tr~X)P \in [I,B(H)]$, and thus $X \in F + [I,B(H)]$.
\end{proof}

\noindent The equivalence of (iii) and (iv$'$) also follows from \cite[Theorem 5.11(iii)]{DFWW}. 

A special case is when $I=J$ is a principal ideal, and then Proposition \ref{P:omega notin I, codim, TFAE} subsumes \cite[Corollary 5.19]{DFWW}. 
Another special case is when $J=\mathscr L_1$, i.e., 
\[\mathscr L_1 \cap [I,B(H)] = \{X\in \mathscr L_1 \mid Tr~X=0\} ~\text{if and only if}~ \mathscr L_1 = ~_{a_{\infty}}I ~(\text{since}~ _{a_{\infty}}I\subset \mathscr L_1),\] 
which by
Corollary \ref{C: ainftyI = L1 if and only if se(omega) subset I}(i) is equivalent to the condition $se(\omega) \subset I$.

The analog below of Proposition \ref{P:omega notin I, codim, TFAE} for the case when $\omega \in \Sigma(I)$ is simpler and its proof is left to the reader. 
The equivalence of (iii) and (iii$'$) is a simple consequence of the five chain of inclusions presented in Section 2.

\newpage

\begin{proposition}\label{P:omega in I, codim, TFAE} For ideals $I$ and $J$, if $\omega \in \Sigma(I)$, then the following conditions are equivalent.
\item[(i)] $dim~\frac{J+[I,B(H)]}{[I,B(H)]} = 0$
\item[(ii)] $J\subset [I,B(H)]$
\item[(iii)] $J \subset ~_aI$
\item[(iii$'$)] $J_a \subset I$
\end{proposition}

Proposition \ref{P:F+[]} and Corollary \ref{C:(F+[])+} allow us to characterize the ideals with $\omega \notin \Sigma(I)$ that support a unique trace up to scalar multiples. 

\begin{theorem}\label{T:i-iii} 
If $I\ne \{0\}$ is an ideal where $\omega \notin \Sigma(I)$, then the following are equivalent.
\item[(i)]  $I$ supports a nonzero trace unique up to scalar multiples.
\item[(ii)]  $I \subset \mathscr L_1$ and every trace on $I$ is a scalar multiple of $Tr$.
\item[(iii)] $dim ~\frac{I}{[I,B(H)]} = 1$
\item[(iv)] $I = F + [I,B(H)]$
\item[(v)]  $I \subset \mathscr L_1$ and $[I,B(H)]=\{X\in I \mid Tr\,X=0\}$
\item[(vi)]  $I$ is am-$\infty$ stable, i.e., $I=~_{a_\infty}I$.
\end{theorem}

\begin{proof}
The equivalence of (iii)-(vi) is the case $J=I$ in Proposition \ref{P:omega notin I, codim, TFAE}. 
The equivalence of (i) and (iv) follows from the case $J=F$ and $\tau = Tr$ in Proposition \ref{P:trace extension} which provides both the existence of a nonsingular trace on $I$ and the condition for its uniqueness. 
Since $_{a_\infty}I \subset \mathscr L_1$, 
(i) and (vi) imply (ii) and (ii) trivially implies (i). 
\end{proof}

\begin{remark}\label{R:IMAX}
By the remarks following Definition \ref{D: stabilizers}, 
$st_{a_\infty}(\mathscr{L}_1)$
is the largest am-$\infty$ stable ideal, so by Theorem \ref{T:i-iii}, 
it is the largest ideal not containing $(\omega)$ that has a nonzero trace unique up to scalar multiples. 
We do not know whether or not an ideal $I$ containing $(\omega)$ can have a nonzero trace unique up to scalar multiples. 
However, nonuniqueness  
for large classes of ideals containing $(\omega)$ follows from Theorems \ref{T: BASICS: seJ notin F+[I,B(H)] implies uncountable dimension-(J+[I,B(H)])/(F+[I,B(H)])}, 
\ref{T: stabilizer containments and uncountable dimension}, 
Corollary \ref{C: soft-edged or soft-complemented}, and Theorem \ref{T: false converse for T: se or sc inf codim implies}. 
\end{remark}

As mentioned in the introduction, a principal ideal $(\xi)$ supports \textit{no nonzero trace} precisely when $\xi$ is 
\textit{regular}.
A similar characterization of the principal ideals supporting a \textit{unique nonzero trace} in terms of 
\textit{regularity at infinity} was obtained in \cite[Corollary 5.6]{mW02}. It is also an immediate consequence of Theorem \ref{T:i-iii}.

\begin{corollary}\label{C:pi supports unique trace iff average at infinity in pi}
Let $\xi \in \text{c}_{\text{o}}^*$ and $\omega \notin \Sigma((\xi))$. 
Then $(\xi)$ supports a nonzero trace unique up to scalar multiples if and only if 
$\xi$ is $\infty$-regular. 
\end{corollary}

As remarked after Definition \ref{D:am-infty stability/regularity}, 
a sequence $\xi \in (\ell^1)^*$ is $\infty$-regular precisely when $(\xi)=(\xi_{a_\infty})$ or, equivalently, 
$\xi_{a_\infty} = O(\xi)$ (see Theorem \ref{T: Theorem 4.12}). 
Moreover, by Remark \ref{R:IMAX} such a sequence must be contained in 
$\Sigma(st_{a_\infty}(\mathscr{L}_1))$ and hence $\sum_{n=1}^\infty \xi_n \,log^m n < \infty$ for every $m$ (see remarks succeeding Proposition \ref{P: stabilizers with logs}).

\newpage

\section{\leftline{\bf Infinite Codimension}}

In this section we present some conditions under which $[I,B(H)]$ has infinite codimension in $I$.
First notice that by setting $I=J$ in the identity in the proof of (i) $\Leftrightarrow$ (ii) 
in Proposition \ref{P:omega notin I, codim, TFAE},
$[I,B(H)]$ has minimal codimension in $I$ precisely when $I = F+[I,B(H)]$. 

Thus if $\omega \in \Sigma(I)$, the codimension is zero precisely when $I$ is am-stable (Theorem \ref{T:DFWW}), 
and if $\omega \notin \Sigma(I)$, the codimension is one precisely when $I$ is am-$\infty$ stable (Theorem \ref{T:i-iii}).
We conjecture that in all other cases, i.e., whenever $I \ne F+[I,B(H)]$,
the codimension of $[I,B(H)]$ in $I$ is infinite, i.e., that 
\[
dim\frac{I}{[I,B(H)]} \in 
\begin{cases} 
\{1,\infty\}	&\text{when}~\omega \notin \Sigma(I) \\
\{0,\infty\}	&\text{when}~\omega \in \Sigma(I).
\end{cases}
\]

In order to verify this conjecture for various classes of ideals, 
we depend on the following result.

\begin{theorem}\label{T: BASICS: seJ notin F+[I,B(H)] implies uncountable dimension-(J+[I,B(H)])/(F+[I,B(H)])} 
If $I$ and $J$ are ideals and $seJ \not\subset F+[I,B(H)]$ then 
$\frac{J+[I,B(H)]}{F+[I,B(H)]}$ has uncountable dimension.
  
In particular, if $seI \not\subset F+[I,B(H)]$  
then $\frac{I}{[I,B(H)]}$ has uncountable dimension.
\end{theorem}

\begin{proof}
Since $seJ = span\,(seJ)^+$ and $F+[I,B(H)]$ is a linear space, it follows that
$(seJ)^+ \not\subset F+[I,B(H)]$. 
Thus, let $X \in (seJ)^+ \setminus (F+[I,B(H)])$ and let $\eta = s(X)$ be the sequence of s-numbers of $X$.  
Since $X-UXU^* \in [I,B(H)]$ for every unitary $U$, 
$diag~\eta \in (seJ)^+ \setminus (F+[I,B(H)])$. 
By definition, $\eta = o(\xi)$ for some 
$\xi \in \Sigma(J)$ 
and without loss of generality assume that $\eta \leq \xi$. 
Also, $\eta_n > 0$ for all $n$ since $diag~\eta \notin F$. 
Define $\zeta(t) := \xi^t\eta^{1-t}$ for $t \in [0,1]$. 
Then since $\zeta(t) \in c_o^*$ and $\zeta(t) \leq \xi$, also $\zeta(t) \in \Sigma(J)$. 
We claim that for any choice of $0=t_0<t_1<\cdots<t_N=1$ the cosets $\{diag~\zeta(t_j)+F+[I,B(H)]\}_{j=0}^{N}$ are linearly independent in 
$\frac{J+[I,B(H)]}{F+[I,B(H)]}$. 
Indeed, assuming otherwise, 
$diag~(\zeta(t_j)+\sum_{i=0}^{j-1} \lambda_i \zeta(t_i)) \in F+[I,B(H)]$ 
for some $0<j \leq N$ and some constants $\lambda_i$, $i=0,1,\dots,j-1$.
Since $F+[I,B(H)]$ is a selfadjoint linear space and $\zeta(t_i)$ are real-valued 
sequences, one can choose all $\lambda_i$ to be real. 
Define $\rho = \zeta(t_j)+\sum_{i=0}^{j-1} \lambda_i \zeta(t_i)$ and set $\chi=\max(\rho,\eta)$. 
Since $\zeta(t_i) = o(\zeta(t_j))$ for $i=0,1,\dots,j-1$ and $\eta = o(\zeta(t_j))$, 
one has $\eta = o(\rho)$, so that $\chi_n=\rho_n$ for $n$ large enough. 
Thus $diag~(\rho-\chi) \in F$ and hence $diag~\chi \in F+[I,B(H)]$. 
Since $\chi \geq \eta$ and since $(F+[I,B(H)])^+$ is hereditary by Corollary \ref{C:(F+[])+}(iii), 
it follows that $diag~\eta \in F+[I,B(H)]$, against the hypothesis. 
Thus the cosets $\{diag~\zeta(t_j)+F+[I,B(H)]\}_{j=0}^{N}$ 
are linearly independent and so $\frac{J+[I,B(H)]}{F+[I,B(H)]}$ has uncountable dimension. 
This implies of course that in case $I=J$,  
$\frac{J+[I,B(H)]}{[I,B(H)]}$ also has uncountable dimension.
\end{proof}

Notice that the condition $seJ \not\subset F + [I,B(H)]$ is equivalent to 
\[seJ \not\subset \begin{cases}
_{a_\infty}I &\text{if $\omega \notin \Sigma(I)$} \\
_aI &\text{if $\omega \in \Sigma(I)$} 
\end{cases}.\]
Notice also that if $L \subset J$ is any ideal for which $(L+[I,B(H)])^+$ is hereditary, 
Theorem \ref{T: BASICS: seJ notin F+[I,B(H)] implies uncountable dimension-(J+[I,B(H)])/(F+[I,B(H)])} with the same proof 
remains valid if we substitute $L+[I,B(H)]$ for $F+[I,B(H)]$. 

In the following theorem, conditions (i) and (ii) are expressed in terms of the am-stability (resp., am-$\infty$ stability) of $seI$. 
Recall from 
Propositions \ref{P:TFAE-seI am-stable, scI am-stable, seI in Ipre-a, Ia in scI} 
and \ref{P: TFAE:seI am-stable, scI am-stable, seI in, scI contains} 
that this is equivalent to the am-stability (resp., am-$\infty$ stability) of $scI$.
Recall also that $st^a(\mathscr L_1)$ is the smallest am-stable ideal and that $st_{a_\infty}(\mathscr L_1)$ is the largest am-$\infty$ stable ideal 
(Definition \ref{D: stabilizers} and succeeding remarks).

\begin{theorem}\label{T: stabilizer containments and uncountable dimension}
Let $I \ne \{0\}$ be an ideal. 
Then $\frac{I}{[I,B(H)]}$ has uncountable 
dimension if any of the following conditions hold.
\item[(i)] $I \subset st_{a_\infty}(\mathscr L_1)$ and $seI$ is not am-$\infty$ stable.
\item[(ii)] $I \supset st^a(\mathscr L_1)$ and $seI$ is not am-stable.
\item[(iii)] $I \not\subset st_{a_\infty}(\mathscr L_1)$ and $I \not\supset st^a(\mathscr L_1)$.
\end{theorem}

\begin{proof}
\item[(i)] 
Since $\omega \notin \Sigma(I)$ because 
$st_{a_\infty}(\mathscr L_1) \subset \mathscr L_1$ by Proposition \ref{P: TFAE:seI am-stable, scI am-stable, seI in, scI contains}, 
$seI \not\subset F+[I,B(H)]$ and the conclusion follows from 
Theorem \ref{T: BASICS: seJ notin F+[I,B(H)] implies uncountable dimension-(J+[I,B(H)])/(F+[I,B(H)])}.
\item[(ii)] $F\subset [I, B(H)]$ since $\omega \in \Sigma(I)$, and thus $seI \not\subset F+[I,B(H)]$, 
by Proposition \ref{P:TFAE-seI am-stable, scI am-stable, seI in Ipre-a, Ia in scI}, 
and the conclusion follows again from 
Theorem \ref{T: BASICS: seJ notin F+[I,B(H)] implies uncountable dimension-(J+[I,B(H)])/(F+[I,B(H)])}.
\item[(iii)] Assume first that $\omega \notin \Sigma(I)$. 
As $I \not\subset st_{a_\infty}(\mathscr L_1)$,  one has
$scI \not\subset st_{a_\infty}(\mathscr L_1)$. 
Since $st_{a_\infty}(\mathscr L_1)$ is the largest am-$\infty$ ideal, $scI$ is not am-$\infty$ stable, 
hence by Proposition \ref{P: TFAE:seI am-stable, scI am-stable, seI in, scI contains}, $seI \not\subset ~_{a_\infty}I$, 
and so by Corollary \ref{C:(F+[])+}(i), $seI \not\subset F+[I,B(H)]$.

Assume now that $\omega \in \Sigma(I)$.
Since $I \not\supset st^a(\mathscr L_1)$, one has $seI \not\supset st^a(\mathscr L_1)$. 
As $st^a(\mathscr L_1)$ is the smallest am-stable ideal, $seI$ is not am-stable, 
hence by Proposition \ref{P:TFAE-seI am-stable, scI am-stable, seI in Ipre-a, Ia in scI}, 
$seI \not\subset [I,B(H)] = F+[I,B(H)]$.

In either case the result follows now from 
Theorem \ref{T: BASICS: seJ notin F+[I,B(H)] implies uncountable dimension-(J+[I,B(H)])/(F+[I,B(H)])}.
\end{proof}

Theorem \ref{T: stabilizer containments and uncountable dimension}(ii),(iii) were motivated by an analysis of ideals of the form \\
$I = se(\xi_a)+ (\xi)$ when $\xi$ is irregular and nonsummable. 

We can extend the method of Theorem \ref{T: stabilizer containments and uncountable dimension} in two directions.

\begin{corollary}\label{C:am/am-infty uncountable dimension} 
Let $I \ne \{0\}$ be an ideal. 
Then $\frac{I}{[I,B(H)]}$ has uncountable dimension if any of the following conditions hold.
\item[(i)] $I \not\supset st^a(\mathscr L_1)$ and 
$I \subset J$ but $I\not\subset st_{a_\infty}(J)$ for some soft-complemented ideal $J$.
\item[(ii)] $I\not\subset st_{a_\infty}(\mathscr L_1)$ and 
$J \subset I$ but $st^a(J)\not\subset I$ for some soft-edged ideal $J$.
\end{corollary}
\begin{proof} 
\item[(i)] 
By Theorem \ref{T: stabilizer containments and uncountable dimension}(iii) it remains to consider the case that $I\subset st_{a_\infty}(\mathscr L_1)$. 
Since $I\not\subset st_{a_\infty}(J)$, then for some $n\geq0$, $I \subset ~_{a_\infty^n}J$ but $I \not\subset ~_{a_\infty^{n+1}}J$. 
And since  $_{a_\infty^n}J$ is soft-complemented as well by Lemma \ref{L: sc(ainftyI),(seI)ainfty}(i), assume without loss of generality that $n=0$, 
i.e., $I \not\subset ~_{a_\infty}J$. 
But then $scI \not\subset ~_{a_\infty}J$ while $scI \subset J$, hence $_{a_\infty}(scI) \subset ~_{a_\infty}J$, so $scI$ is not am-$\infty$ stable. 
Hence the conclusion follows from Theorem \ref{T: stabilizer containments and uncountable dimension}(i) and Proposition \ref{P: TFAE:seI am-stable, scI am-stable, seI in, scI contains}.
\item[(ii)] By Theorem \ref{T: stabilizer containments and uncountable dimension}(iii) it remains to consider the case that $I \supset st^a(\mathscr L_1)$ and that for some $n\geq0$, $J_{a^n} \subset I$ but $J_{a^{n+1}} \not\subset I$. 
In particular, this implies that $J_{a^n} \not\subset \mathscr L_1$ since otherwise $st^a(J) \subset st^a(\mathscr L_1) \subset I$ against the hypothesis.
Hence, since $J$ is soft-edged, $J_{a^n}$ too is soft-edged by Lemma \ref{L:se,sc,pre-am,am}(ii$'$). 
Thus $J_{a^n} \subset seI$, and hence $J_{a^{n+1}} \subset (seI)_a$ while $J_{a^{n+1}} \not\subset seI$. 
This shows that seI is not am-stable and hence the conclusion follows from Theorem \ref{T: stabilizer containments and uncountable dimension}(ii).
\end{proof}

Using a method similar to the one employed in Theorem \ref{T: stabilizer containments and uncountable dimension} and building on Propositions \ref{P:omega notin I, codim, TFAE} and \ref{P:omega in I, codim, TFAE} we obtain:

\begin{proposition}\label{P:I,J uncountable dimension} 
If $I$ and $J$ are nonzero ideals and if $I$ is soft complemented or $J$ is soft-edged, then
\[
dim~\frac{J+[I,B(H)]}{[I,B(H)]} \quad \text{is}\quad 
\begin{cases}
1 ~\text{or uncountable}	&\text{if $\omega \notin \Sigma(I)$}\\
0 ~\text{or uncountable}	&\text{if $\omega \in \Sigma(I)$}
\end{cases}.
\]
\end{proposition}

\begin{proof} Assume that $\omega \notin \Sigma(I)$. 
If $dim~\frac{J+[I,B(H)]}{[I,B(H)]} \ne 1$, then by 
Proposition \ref{P:omega notin I, codim, TFAE}, $J \not\subset~ _{a_\infty}I$. 
If $J$ is soft-edged, then $seJ \not\subset~ _{a_\infty}I$. 
If $I$ is soft-complemented, then  $se J \subset ~_{a_\infty}I$ would imply that 
$J \subset scJ = sc(se(J)) \subset  sc(_{a_\infty}I) = ~_{a_\infty}(scI) = ~_{a_\infty}I$ (see Lemma \ref{L: sc(ainftyI),(seI)ainfty}(i)).
Thus, in either case $se J \not\subset ~_{a_\infty}I$, which is equivalent to 
$se J \not\subset F + [I, B(H)]$ by Proposition \ref{P:omega notin I, codim, TFAE}.
Uncountable dimension then follows from Theorem \ref{T: BASICS: seJ notin F+[I,B(H)] implies uncountable dimension-(J+[I,B(H)])/(F+[I,B(H)])}. 

The case when $\omega \in \Sigma(I)$ (i.e., $F \subset [I,B(H)]$) follows similarly from 
Proposition \ref{P:omega in I, codim, TFAE} and Lemma \ref{L:se,sc,pre-am,am}(i$'$).
\end{proof}

A case of special interest is when $I = J$ which, for soft-edged and soft-complemented ideals, proves the codimension conjecture stated in the introduction.

\begin{corollary}\label{C: soft-edged or soft-complemented}  
If $I$ is a soft-edged or soft-complemented ideal, then 
\[
dim \,\frac{I}{[I,B(H)]} \quad \text{is} \quad 
\begin{cases}
1 \quad \text{or uncountable} \quad	&\text{if $\omega \notin \Sigma(I)$} \\
0 \quad \text{or uncountable} \quad	&\text{if $\omega \in \Sigma(I)$.}
\end{cases}
\]
\end{corollary}

\noindent In particular, $dim \,\frac{(\omega)}{[(\omega),B(H)]}$ is uncountable since $\omega$ is not regular. 

Another case of interest is when $I$ or $J$ are the trace class $\mathscr L_1$, 
which is both soft-edged and soft-complemented. 

\begin{corollary}\label{C:dimensions of 3 quotients} Let $I$ be a nonzero ideal. Then
\item[(i)] 
\[
dim \,\frac{I+[\mathscr L_1,B(H)]}{[\mathscr L_1,B(H)]} \quad \text{is} \quad 
\begin{cases}
1 	&\text{if $I \subset ~_{a_\infty}(\mathscr L_1)$} \\
\text{uncountable}	&\text{if $I \not\subset ~_{a_\infty}(\mathscr L_1)$}
\end{cases}
\]
\item[(ii)] 
\[
dim \,\frac{\mathscr L_1+[I,B(H)]}{[I,B(H)]} \quad \text{is} \quad 
\begin{cases}
0 	&\text{if $\omega \in \Sigma(I)$} \\
1 	&\text{if $\omega \in \Sigma(scI) \setminus \Sigma(I)$} \\
\text{uncountable}	&\text{if $\omega \notin \Sigma(scI)$}
\end{cases}
\]
\item[(iii)] If $\omega \notin \Sigma(I)$ then 
\[
dim \,\frac{I}{I \cap \mathscr L_1+[I,B(H)]} \quad \text{is} \quad 
\begin{cases}
0 	&\text{if $I \subset \mathscr L_1$} \\
\text{uncountable} \quad	&\text{if $I \not\subset \mathscr L_1$}
\end{cases}.
\] 
In particular, if $\mathscr L_1 \not\subset I$, then there are uncountably many linearly independent extensions of $Tr$ from $\mathscr L_1$ to $I$.
\end{corollary}

\newpage

\begin{proof} 
\item[(i)] Immediate from Propositions \ref{P:I,J uncountable dimension} and \ref{P:omega notin I, codim, TFAE} (the equivalence of (i) and (iv)) since $\omega \notin \Sigma(\mathscr L_1)$. 
\item[(ii)] Recall that $\mathscr L_1 \subset [I,B(H)]$ if and only if $\omega \in \Sigma(I)$. 
If $\omega \notin \Sigma(I)$, by Proposition \ref{P:omega notin I, codim, TFAE}, 
$dim \,\frac{\mathscr L_1+[I,B(H)]}{[I,B(H)]} = 1$ when $\mathscr L_1 \subset ~_{a_\infty}I$, 
which by Corollary \ref{C: ainftyI = L1 if and only if se(omega) subset I}(i) is equivalent to 
$\omega \in \Sigma(scI)$, and $dim \,\frac{\mathscr L_1+[I,B(H)]}{[I,B(H)]}$ is uncountable otherwise.
\item[(iii)] First notice that if $seI \subset I \cap \mathscr L_1 + [I, B(H)]$, 
then 
\[(seI)^+ \subset (I \cap \mathscr L_1 + [I, B(H)])^+ = (I \cap \mathscr L_1)^+ \subset \mathscr L_1^+\] 
and hence $seI  \subset \mathscr L_1$, where the equality follows from Proposition \ref{P:(L1+[])+=L1+}. 
Since $\mathscr L_1$ is soft-complemented, $I \subset scI = sc(seI) \subset \mathscr L_1$.
Thus, if $I \not\subset \mathscr L_1$, $seI \not\subset I \cap \mathscr L_1+[I,B(H)]$. 
By the second remark following Theorem \ref{T: BASICS: seJ notin F+[I,B(H)] implies uncountable dimension-(J+[I,B(H)])/(F+[I,B(H)])}, 
$dim \,\frac{I}{I \cap \mathscr L_1+[I,B(H)]}$ is uncountable. 
The particular case is then clear.
\end{proof}

\begin{remark}\label{R: Dixmier proved}
Dixmier proved in \cite{jD66} that the am-closure $(\eta)^- = ~_a(\eta_a)$ of a principal ideal $(\eta)$ 
for which $(\eta) \subset se(\eta)_a$ (i.e., $\eta = o(\eta_a)$, which is equivalent to $\frac{(\eta_a)_{2n}}{(\eta_a)_n} \rightarrow \frac{1}{2}$) supports a positive singular trace. 
In \cite[Section 5.27 Remark 1]{DFWW} it was noted that Dixmier's construction can be used to show that $dim~\frac{(\eta)^-}{cl[(\eta)^-,B(H)]} = \infty$ 
where $cl$ denotes the closure in the principal ideal norm. 
Corollary \ref{C: soft-edged or soft-complemented} shows that $dim~\frac{(\eta)^-}{[(\eta)^-,B(H)]}$ is uncountable 
follows from the weaker hypothesis $\eta_a \ne O(\eta)$, 
i.e., $\eta$ is not regular.
Indeed, by \cite[Theorem 5.20]{DFWW} (see also \cite{vKgW02}), $(\eta)$ is not am-stable 
precisely when $(\eta)^-$ is not am-stable, 
i.e., $dim \,\frac{(\eta)^-}{[(\eta)^-,B(H)]} > 0$. 
In \cite{vKgW07-soft} we show that $(\eta)^-$ 
is always soft-complemented and so by Corollary \ref{C: soft-edged or soft-complemented}, $dim\, \frac{(\eta)^-}{[(\eta)^-,B(H)]}$ is uncountable. 
\end{remark}

The condition in Theorem \ref{T: stabilizer containments and uncountable dimension}(ii) for $[I,B(H)]$ to have infinite 
codimension in $I$, namely that $seI$ be not am-stable, is only sufficient. 
The next 
theorem presents a class of ideals $I$ with $seI$ 
am-stable but 
$dim \frac{I}{[I,B(H)]} = \infty$. 
The technique used does not depend on 
Theorem \ref{T: BASICS: seJ notin F+[I,B(H)] implies uncountable dimension-(J+[I,B(H)])/(F+[I,B(H)])} 
nor 
on the method of its proof but is more combinatoric in nature. 
As indicated 
in Corollary \ref{C: false converse for T: se or sc inf codim implies+}, 
this technique can be used to prove infinite codimension for a wider class of ideals.

\begin{theorem}\label{T: false converse for T: se or sc inf codim implies}
For every am-stable principal ideal $J \ne \{0\}$ there is an ideal $I$ with $se J \subset I \subset J$ for which $dim~\frac{I}{[I,B(H)]} = \infty$, yet $seI$ and hence $scI$ are am-stable. 
\end{theorem}

\begin{proof}
Choose a generator $\mu$ for $\Sigma(J)$. 
Then $\mu$ is regular, i.e., $\mu \asymp \mu_a$, and hence nonsummable. 
Now construct a sequence $\xi \in \text{c}_{\text{o}}^*$ 
together with a strictly increasing sequence of indices $<p_l>_{l \in \mathbb N}$ 
for which:
(i) $\xi \leq \mu$, 
(ii) $\xi_{p_l} = \frac{1}{l} \mu_{p_l}$ and
(iii) $(\xi_a)_{p_l} \geq \frac{1}{2} (\mu_a)_{p_l}$. 
Set $p_1 = 1$, $\xi_1 = \mu_1$ 
and assume that $p_1<\cdots<p_l$ and $\xi_i$ for $1 \leq i \leq p_l$ have been chosen so that (i)-(iii) hold.
Define $\xi_i := \min~\{\xi_{p_l},\mu_i\}$ for $p_l < i \leq p_{l+1}-1$ 
where $p_{l+1}> p_l$, $p_{l+1} \geq 3$ is chosen large enough so that 
$\sum_{i=1}^{p_{l+1}-1} \xi_i \geq \frac{3}{4}\sum_{i=1}^{p_{l+1}-1} \mu_i$, which is possible due to the nonsummability of $\mu$.
Define $\xi_{p_{l+1}} := \frac{1}{l+1}~\mu_{p_{l+1}}$. Then 
\[
\sum_{i=1}^{p_{l+1}} \xi_i \geq \frac{3}{4}\sum_{i=1}^{p_{l+1}-1} \mu_i \geq \frac{1}{2}\sum_{i=1}^{p_{l+1}} \mu_i,
\]
the inequality following from $\sum_{i=1}^{p_{l+1}-1} \mu_i \geq 
2\mu_{p_{l+1}}$. 
Therefore (i)-(iii) hold. Notice that (ii),(iii) imply $\xi \not\asymp \xi_a$, i.e., $\xi$ is irregular.

Define now $I := se J+(\xi)$. 
Since $\xi \leq \mu$, one has $(\xi)\subset(\mu)=J$ and hence $seJ \subset I \subset J$.  
Since $J=scJ$ because principal ideals are soft-complemented \cite{vKgW07-soft}, $seJ$ and $J$ form a soft pair and hence $scI=J$ and $seI=seJ$. 
Since $J$ is am-stable, so are $seI$ and $scI$ (see Remark \ref{R: I am-stable implies seI,scI am-stable with false converse}).

Notice that $\mu \notin \ell^1$ implies $\omega = o(\mu_a)$ and hence $\omega \in \Sigma(se J) \subset \Sigma(I)$.
Condition (iii) implies that $\xi \notin \ell^1$.

The pair $\xi,I$ has the following property which as we will show does imply that $dim~\frac{I}{[I,B(H)]} = \infty$: there is a strictly increasing sequence of indices $<p_l>_{l \in \mathbb N}$ such that 
for every $\chi \in \Sigma(I)$, $(\frac{\chi}{\xi_a})_{mp_l} \rightarrow 0$ for some integer $m \in \mathbb N$.

Indeed, for every $\chi \in \Sigma(I)$ there are $\alpha \in \text{c}_{\text{o}}^*$, $M>0$ and 
$m \in \mathbb N$ for which $\chi \leq \alpha \mu + MD_m\xi$ and so
\begin{align*}
(\frac{\chi}{\xi_a})_{mp_l} &\leq \frac{(\alpha \mu)_{mp_l} + M(D_m \xi)_{mp_l}}{(\xi_a)_{mp_l}} \\
&\leq \frac{\alpha_{p_l} \mu_{p_l} + M \xi_{p_l}}{(\xi_a)_{mp_l}}  \qquad\qquad \text{by the monotonicity of $\alpha$ and $\mu$} \\
&	\hspace{1.85in}					 \text{and the definition of $D_m$} \\
&\leq m~ \frac{\alpha_{p_l} \mu_{p_l} + M \xi_{p_l}}{(\xi_a)_{p_l}}  \quad\qquad \text{since $(\xi_a)_{mp_l} \geq \frac{1}{m} (\xi_a)_{p_l}$} \\
&\leq 2m~ \frac{\alpha_{p_l} \mu_{p_l} + \frac{M}{l} \mu_{p_l}}{(\mu_a)_{p_l}}  \qquad \text{by (ii) and (iii)} \\
&\leq 2m (\alpha_{p_l} + \frac{M}{l}) \rightarrow 0 \qquad \text{since $\mu_a \geq \mu$.} 
\end{align*}

We proceed now to prove that the codimension of $[I,B(H)]$ is 
infinite. 
For each positive integer $N>1$
and $1 \leq j \leq N$, 
choose strictly increasing sequences of indices $m_k^{(j)} < n_k^{(j)}$ 
where for
all $k \in \mathbb N$:
\item[(a)] $n_k^{(j)} \in \{p_l \}$ and when $n_k^{(j)} = p_l$ then $l \geq k$.
\item[(b)] $\sum_{i=m_k^{(j)}}^{n_k^{(j)}} \xi_i \geq 3~\sum_{i=1}^{m_k^{(j)}} \xi_i$.
\item[(c)] $m_k^{(j-1)}= kn_k^{(j)}+1$ for $2\leq j \leq N$.
\item[(d)] $m_{k+1}^{(N)}=\min~\{i \mid \xi_{kn_k^{(1)}} \geq N \xi_i\}$. 

\noindent To construct the sequences $m_k^{(j)}$ and $n_k^{(j)}$, start with $m_1^{(N)}=1$ and choose $n_1^{(N)} \in \{p_l\}$ satisfying (b), 
which 
is possible since $\xi \notin \ell^1$. 
Then set $~m_1^{(N-1)}= n_1^{(N)}+1$ 
according to (c). 
Alternating between (b) and (c), 
obtain $m_1^{(j)},n_1^{(j)}$ for $1 \leq j \leq N$. 
Then choose $m_2^{(N)} > n_1^{(1)}$ so to satisfy (d), which is possible since $\xi \in \text{c}_{\text{o}}^*$.
Continue then the construction for $k=2$ and so on. So for each $k$, 
\[
m_k^{(N)} <kn_k^{(N)}<m_k^{(N-1)}<kn_k^{(N-1)}<\dots<m_k^{(1)}<kn_k^{(1)}<m_{k+1}^{(N)}.
\]
Then define $N$ sequences $\eta^{(j)} \in \text{c}_{\text{o}}^*$ by setting 
\[
(\eta^{(j)})_i := 
\begin{cases}
\min~(j,p) \xi_i,					&\text{if} ~m_k^{(p)} \leq i \leq kn_k^{(p)}~\text{for}~1 \leq p \leq N \\
\min~(\xi_{kn_k^{(1)}},j\xi_i),		&\text{if} ~kn_k^{(1)} < i < m_{k+1}^{(N)}. 
\end{cases}
\]
Thus $\xi = \eta^{(1)} \leq \eta^{(2)} \leq \cdots \leq \eta^{(N)}$ and 
$\eta^{(j)} \leq j\xi$ so that $\eta^{(j)} \in \Sigma(I)$ for every $1 \leq j \leq N$. 

To illustrate this construction, the following figure provides a continuous analog of the sequences $\eta^{(1)},\eta^{(2)}$ and $\eta^{(3)}$ for the case $N=3$.

\begin{figure}[h]
\includegraphics{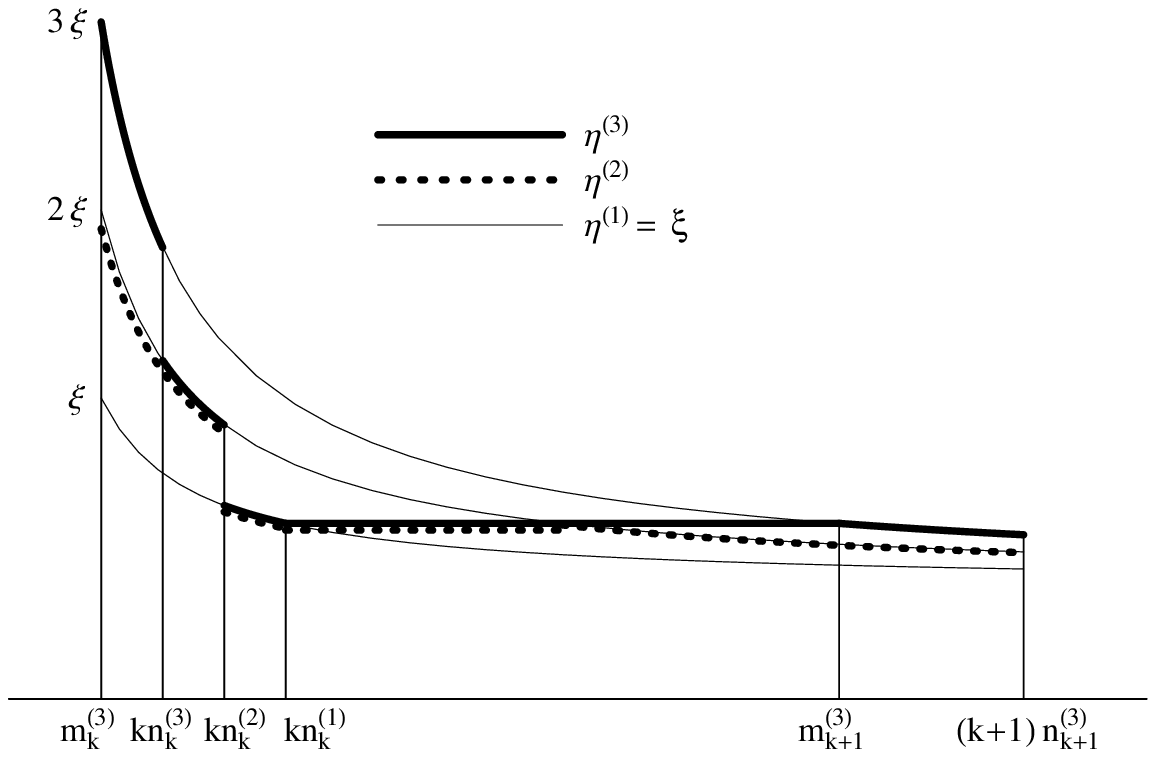}
\end{figure}

To prove that the cosets $diag~(\eta^{(j)})+[I,B(H)]$ are linearly independent, 
suppose $diag~\sum_{j=1}^{N} \lambda_j \eta^{(j)} \in [I,B(H)]$ for some $<\lambda_j> ~\in \mathbb C^N$. 
Since $[I,B(H)]$ is selfadjoint and the sequences $\eta^{(j)}$ are real-valued, 
one can assume without loss of generality that $<\lambda_j> ~\in \mathbb R^N$. 
Define $\zeta := \sum_{j=1}^{N} \lambda_j \eta^{(j)}$, $\gamma_1:=0$, 
$
\gamma_p:= \sum_{j=1}^{p-1} j\lambda_j ~\text{for}~ 2 \leq p \leq N
$
and
$
\beta_p := \sum_{j=p}^{N} \lambda_j ~\text{for}~ 1 \leq p \leq N.
$ 
A direct computation shows that $\zeta_1 = (\gamma_N + N\beta_N)\xi_1$ and that
\[
\zeta_i = 
\begin{cases}
(\gamma_p+p\beta_p) \xi_i,			&\text{if} ~m_k^{(p)} \leq i \leq kn_k^{(p)}~\text{for}~1 \leq p \leq N \\
\gamma_p \xi_i + \beta_p \xi_{kn_k^{(1)}},&\text{if} ~kn_k^{(1)} < i < m_{k+1}^{(N)} ~\text{and} \\
							&(p-1)\xi_i \leq \xi_{kn_k^{(1)}} < p \xi_i~\text{for}~2 \leq p \leq N.
\end{cases}
\]
Setting $\gamma = \max~|\gamma_p+p\beta_p|$, we claim that $|\zeta_i| \leq \gamma \xi_i$ for every $i$.
If $m_k^{(p)} \leq i \leq km_k^{(p)}$ for some $k \geq 1$ and some $1 \leq p \leq N$, then $|\zeta_i|=|\gamma_p+p\beta_p|\xi_i \leq \gamma \xi_i$.
And if $kn_k^{(1)} < i < m_{k+1}^{(N)}$ and $(p-1)\xi_i \leq \xi_{kn_k^{(1)}} < p \xi_i$ for some $k \geq 1$ and some $2 \leq p \leq N$, 
then $\zeta_i = \gamma_p \xi_i + \beta_p \xi_{kn_k^{(1)}}$.
Assuming that $\zeta_i \geq 0$, one has
\[
0 \leq \zeta_i \leq 
\begin{cases}
(\gamma_p+p\beta_p) \xi_i,							&\text{if}  ~\beta_p \geq 0 \\
(\gamma_p  + (p-1)\beta_p) \xi_i=(\gamma_{p-1} + (p-1)\beta_{p-1})\xi_i,	&\text{if} ~\beta_p < 0. 
\end{cases}
\]
In either case, $\zeta_i \leq \gamma \xi_i$. 
In the case that $\zeta_i < 0$, the same argument can be applied to $-\zeta_i$ to obtain the claim. 

The crux of the proof is to show that $\gamma =0$, whence an elementary computation will show
that the linear system of the $N$ equations 
$\gamma_p+p\beta_p = 0$ has only the trivial solution $\lambda_1=\lambda_2=\cdots=\lambda_N=0$, i.e., the $N$ cosets are linear independent.
To prove that $\gamma = 0$, 
first choose $1 \leq r \leq N$ for which $\gamma = |\gamma_r+r\beta_r|$.
Then for every $m_k^{(r)} \leq n \leq kn_k^{(r)}$ and every $n'\geq n$, 
one has $|\zeta_{n'}| \leq \gamma \xi_{n'} \leq \gamma \xi_n = |\zeta_n|$. 
This implies that  if $\pi: \mathbb N \rightarrow \mathbb N$ is an injection for which $|\zeta_\pi|$ is monotone nonincreasing, then 
$\{\zeta_i \mid i = m_k^{(r)},\dots,n \} \subset \{\zeta_{\pi(i)} \mid i=1,2,\dots,n\}$.
Define the set \begin{center}$\Lambda_n := \{ \pi(i) \mid i=1,2,\cdots,n \} \setminus \{ m_k^{(r)}, \cdots,n\}$. \end{center}
Then $card~\Lambda_n = m_k^{(r)}-1$ and, from the monotonicity of $\xi$ and the fact that $|\zeta_i|\leq \gamma\xi_i$ for every $i$,
\begin{align*}
|\sum_{i=1}^{n} \zeta_{\pi(i)}| &= |\sum_{i=m_k^{(r)}}^{n} \zeta_i + \sum_{i \in \Lambda_n} \zeta_i| 
\geq |\sum_{i=m_k^{(r)}}^{n} \zeta_i| - \sum_{i \in \Lambda_n} |\zeta_i| \\
&= \gamma\sum_{i=m_k^{(r)}}^{n} \xi_i - \sum_{i \in \Lambda_n} |\zeta_i| 
\geq \gamma (\sum_{i=m_k^{(r)}}^{n} \xi_i - \sum_{i \in \Lambda_n} \xi_i) \\
&\geq \gamma (\sum_{i=m_k^{(r)}}^{n} \xi_i - \sum_{i=1}^{m_k^{(r)}-1} \xi_i).
\end{align*}
If additionally $n_k^{(r)} \leq n \leq kn_k^{(r)}$,
then combining this with the inequality in (b) yields that 
$|\sum_{1}^{n} \zeta_{\pi(i)}| \geq \frac{\gamma}{2} \sum_{1}^{n} \xi_i$, 
that is, $|((\zeta_{\pi})_a)_n| \geq \frac{\gamma}{2}(\xi_a)_n$.
The assumption that $diag ~ \zeta \in [I,B(H)]$ implies, by Theorem \ref{T:DFWW}, that 
$|(\zeta_\pi)_a| \leq \rho$ for some $\rho \in \Sigma(I)$.
But then, by the first part of this proof, there exists an $m \in \mathbb N$ for which 
$(\frac{\rho}{\xi_a})_{mp_l} \rightarrow 0$.
As $(\frac{\rho}{\xi_a})_{mn_k^{(r)}}\geq \frac{\gamma}{2}$ for all $k \geq m$ and as $n_k^{(r)} \in \{p_l\}$, it follows that $\gamma=0$, which concludes the proof.
\end{proof}

In contrast to the other results on codimension in this paper, the proof of Theorem \ref{T: false converse for T: se or sc inf codim implies} does not seem to yield uncountable codimension.

\begin{corollary}\label{C: false converse for T: se or sc inf codim implies+}
The second part of the proof of Theorem \ref{T: false converse for T: se or sc inf codim implies} shows that 
its conclusion holds for a larger class: 
if $I$ is an ideal for which there exists a nonsummable sequence $\xi \in \Sigma(I)$ and  
a monotone sequence of indices $\{p_l\}$ so that for every $\chi \in \Sigma(I)$ there is an associated 
$m \in \mathbb N$ for which $(\frac{\chi}{\xi_a})_{mp_l} \rightarrow 0$, then $dim\, \frac{I}{[I,B(H)]} = \infty$. 
\end{corollary}

Acknowledgments. We wish to thank Daniel Beltita for his input on this paper, and Ken Davidson and Allan Donsig for their early input on the research.


\begin{thebibliography}{99}

\bibitem{AGPS87}
Albeverio, S., Guido, D., Ponosov, A., and Scarlatti, S., 
\emph{Singular traces and compact operators,} J. Funct. Anal. \textbf{137}(2) (1996), pp.~281--302. MR1387512 (97j:46063)

\bibitem{Aljan-Aran77}
	{Aljan{\v{c}}i\'c}, S. and {Arandelovi\'c}, D.,
	$O$-regularly varying functions,
	Publ. Inst. Math. (Beograd) (N.S.)~
	{\bf 22(36)}~(1977), 5--22.
\qquad MR~57 \#6317.

\bibitem{jA77}
Anderson, J., \emph{Commutators of compact operators,} J. Reine Angew. Math. \textbf{291} (1977), pp.~128--132.

\bibitem{nB64}
Bary, N. K., \emph{A treatise on trigonometric series. {V}ol. {II},} Authorized translation by Margaret F. Mullins, A Pergamon Press Book, The Macmillan Co. (1964).

\bibitem{BGT89}
Bingham, N. H., Goldie, C. M., and Teugels, J. L., 
\textit{Regular Variation,} Encyclopedia of Mathematics and its Applications (27), Cambridge University Press, 1989.  MR1015093 (90i:26003)

\bibitem{BP65}
Brown, A. and Pearcy, C., \emph{Structure of commutators of operators,} Ann. of Math. (2) \textbf{82} (1965), pp.~112--127.

\bibitem{jC41}
Calkin, J. W., \emph{Two-sided ideals and congruences in the ring of bounded operators in {H}ilbert space,} 
Ann. of Math. (2) \textbf{42} (1941), pp.~839--873.

\bibitem{aC82}
Connes, A., \emph{Noncommutative differential geometry Part I, the Chern character in $K$-homology,} 
Inst. Hautes Etudes Sci., Bures-Sur-Yvette, 1982.

\bibitem{aC83}
Connes, A., \emph{Noncommutative differential geometry Part II, de Rham homology and noncommutative algebra,} 
Inst. Hautes Etudes Sci., Bures-Sur-Yvette, 1983.

\bibitem{aC85}
Connes, A., \emph{Noncommutative differential geometry,} 
Inst. Hautes \'Etudes Sci. Publ. Math. \textbf{62} (1985), pp.~257--360. MR87i:58162

\bibitem{jD66}
Dixmier, J., \emph{Existence de traces non normales,} 
C. R. Acad. Sci. Paris S\'er. A-B \textbf{262} (1966), pp.~A1107--A1108.

\bibitem{kDnK98}
Dykema, K. J. and Kalton, N. J., \emph{Spectral characterization of sums of commutators. {I}{I},} J. Reine Angew. Math. \textbf{504} (1998), pp.~127--137.

\bibitem{DFWW}
Dykema, K., Figiel, T., Weiss, G., and Wodzicki, M., \emph{The commutator structure of operator ideals,} 
Adv. Math., 185/1 (2004), pp. 1--79.


\bibitem{kDgWmW00}
Dykema, K., Weiss, G., and Wodzicki, M., \emph{Unitarily invariant trace extensions beyond the trace class,} 
Complex analysis and related topics (Cuernavaca, 1996), Oper. Theory Adv. Appl. \textbf{114} (2000), pp.~59--65.

\bibitem{bF51}
Fuglede, B., \emph{A commutativity theorem for normal operators,} Proc. Nat. Acad. Sci. U. S. A. \textbf{36} (1950), pp.~35--40.

\bibitem{GK69}
Gohberg, I. C. and Kre{\u\i}n, M. G., \emph{Introduction to the theory of linear nonselfadjoint operators.} American Mathematical Society (1969).

\bibitem{pH54}
Halmos, P. R., \emph{Commutators of operators. {II},} Amer. J. Math. \textbf{76} (1954), pp.~191--198.

\bibitem{pH82}
Halmos, P. R., \emph{A Hilbert space problem book,} 2nd Edition, Graduate Texts in Mathematics (19), Springer-Verlag (1982). From 1st Edition, D. Van Nostrand Co., Inc., Princeton, N.J.-Toronto, Ont.-London (1967).

\bibitem{vKgW02}
Kaftal, V. and Weiss, G., \emph{Traces, ideals, and arithmetic means,} Proc. Natl. Acad. Sci. USA \textbf{99} (11) (2002), pp.~7356--7360.

\bibitem{vKgW07-soft}
Kaftal, V. and Weiss, G.,\emph{Soft ideals and arithmetic mean ideals,} IEOT, to appear.

\bibitem{vKgW07-2nd}
\emph{Second order arithmetic means in operator ideals},  Operators and Matrices, to appear

\bibitem{vKgW07-OT21}
\emph{A survey on the interplay between arithmetic mean ideals, traces, lattices of operator ideals, and an infinite Schur-Horn majorization theorem,} Proceedings of the 21st International Conference on Operator Theory, Timisoara 2006 (Theta Bucharest), to appear

\bibitem{vKgW07-Density}
Kaftal, V. and Weiss, G., \emph{B(H) lattices, density and arithmetic mean ideals}, preprint.

\bibitem{vKgW07-Majorization}
Kaftal, V. and Weiss, G.,\emph{Majorization for infinite sequences, an extension of the Schur-Horn Theorem, and operator ideals,} in preparation.


\bibitem{nK87}
Kalton, N. J., \emph{Unusual traces on operator ideals,} Math. Nachr. \textbf{134} (1987), pp.~119--130. MR89a:47070

\bibitem{nK89}
Kalton, N. J., \emph{Trace-class operators and commutators,} J. Funct. Anal. \textbf{86} (1989), pp.~41--74.

\bibitem{nK98}
Kalton, N. J., \emph{Spectral characterization of sums of commutators {I},} J. Reine Angew. Math. \textbf{504} (1998), pp.~115--125.

\bibitem{PT71}
Pearcy, C. and Topping, D., \emph{On commutators in ideals of compact operators,} Michigan Math. J. \textbf{18} (1971), pp.~247--252.

\bibitem{cP51}
Putnam, C. R., \emph{On normal operators in Hilbert space,} Amer. J. Math. \textbf{73} (1951), pp.~357--362.

\bibitem{rS60}
Schatten, R., \emph{Norm ideals of completely continuous operators,} Ergebnisse der Mathematik und ihrer Grenzgebiete, 	Neue Folge, Heft~ {\bf 27} , Berlin, Springer-Verlag, 1960.

\bibitem{vS83}
Shul'man, V. S., \emph{Linear equations with normal coefficients,} Dokl. Akad. Nauk SSSR \textbf{270} {(5)} (1983), pp.~1070--1073.

\bibitem{jV89}
Varga, J., \emph{Traces on irregular ideals,} Proc. Amer. Math. Soc. \textbf{107} 3 (1989), pp.~715--723.

\bibitem{dV79/81}
Voiculescu, D., \emph{Some results on norm-ideal perturbations of {H}ilbert space operators {I}{I},} 
J. Operator Theory \textbf{5} (1981), pp.~77--100.

\bibitem{gW75}
Weiss, G., \emph{Commutators and Operators Ideals,} dissertation (1975), University of Michigan Microfilm.

\bibitem{gW83}
Weiss, G., \emph{An extension of the {F}uglede commutativity theorem modulo the
              {H}ilbert-{S}chmidt class to operators of the form $\sum
              {M}\sb{n}{X}{N}\sb{n}$,} Trans. Amer. Math. Soc. \textbf{278} (1) (1983), pp.~1--20.

\bibitem{gW80}
Weiss, G., \emph{Commutators of {H}ilbert-{S}chmidt operators. {I}{I},} IEOT \textbf{3} (4) (1980), pp.~574--600.

\bibitem{gW86}
Weiss, G., \emph{Commutators of {H}ilbert-{S}chmidt operators. {I},} IEOT \textbf{9} (6) (1986), pp.~877--892.

\bibitem{mW94}
Wodzicki, M., \emph{Algebraic {$K$}-theory and functional analysis,} 
First European Congress of Mathematics, Vol.\ II (Paris,1992) \textbf{120} (1994), pp.~485--496. MR97f:46112

\bibitem{mW02}
Wodzicki, M., \emph{Vestigia investiganda,} Mosc. Math. J. \textbf{4} (2002), pp.~769--798, 806. MR1986090
\end{thebibliography}
\end{document}